\pdfoutput=1
\documentclass[12pt]{article}

\usepackage[utf8]{inputenc}
\usepackage[T1]{fontenc}
\usepackage[english]{babel}
\usepackage{amsmath,amssymb,amsthm}
\usepackage{graphicx}
\usepackage{xcolor}
\usepackage{authblk}
\usepackage{enumitem}
\usepackage{cite}
\usepackage{fullpage}
\usepackage{bm}
\usepackage{array,booktabs}
\usepackage{algorithm}
\usepackage{algpseudocode}
\usepackage{subcaption}

\usepackage{tikz}
\usetikzlibrary{arrows.meta,positioning}

\usepackage{caption}
\usepackage{hyperref}
\hypersetup{
	colorlinks=true,
	linkcolor=blue,
	citecolor=blue,
	urlcolor=teal
}
\usepackage{cleveref}

\newtheorem{theorem}{Theorem}[section]
\newtheorem{lemma}[theorem]{Lemma}
\newtheorem{proposition}[theorem]{Proposition}
\newtheorem{corollary}[theorem]{Corollary}
\theoremstyle{definition}
\newtheorem{definition}[theorem]{Definition}
\newtheorem{example}[theorem]{Example}
\newtheorem{remark}[theorem]{Remark}
\newtheorem{method}[theorem]{Method}

\DeclareMathOperator{\supp}{supp}
\DeclareMathOperator{\Span}{span}
\DeclareMathOperator{\diam}{diam}
\DeclareMathOperator{\lev}{lev}
\DeclareMathOperator{\dist}{dist}
\DeclareMathOperator{\core}{core}
\DeclareMathOperator{\Id}{Id}
\DeclareMathOperator\parent{\texttt{parent}}
\DeclareMathOperator\children{\texttt{children}}

\newcommand\QQ{\mathcal{Q}}
\newcommand\QQQ{\mathbb{Q}}
\newcommand\BB{\mathcal{B}}
\newcommand\NN{\mathbb{N}}
\newcommand\AAA{\mathcal{A}}
\newcommand\HH{\mathcal{H}}
\newcommand\MM{\mathcal{M}}
\newcommand\MMM{\mathbb{M}}
\newcommand\EE{\mathcal{E}}
\newcommand{\Muno}[1]{\mathcal{M}^{\scriptstyle \mathrm{OPT}}_{#1}}
\newcommand{\Mdos}[1]{\mathcal{M}^{\scriptstyle \mathrm{SUBOPT}}_{#1}}
\newcommand{\Mtres}[1]{\mathcal{M}^{\scriptstyle \mathrm{AUX}}_{#1}}
\newcommand{\Mcuatro}[1]{\mathcal{M}^{\scriptstyle \mathrm{ADM}}_{#1}}
\newcommand{\MMMuno}[1]{\mathbb{M}^{\scriptstyle \mathrm{OPT}}_{#1}}
\newcommand{\MMMdos}[1]{\mathbb{M}^{\scriptstyle \mathrm{SUBOPT}}_{#1}}

\title{\bf {Adaptive refinement \\based on local approximation power \\of hierarchical spline spaces}}

\author{Gustavo A. Fernandez Lezcano\thanks{Corresponding author: Gustavo A. Fernandez Lezcano (gafernandezl@fiq.unl.edu.ar)} , Eduardo M.~Garau \& Bárbara S. Ivaniszyn}

\affil{\small Universidad Nacional del Litoral, \\
	Consejo Nacional de Investigaciones Cient\'ificas y T\'ecnicas, \\
	FIQ, Santa Fe, Argentina}

\begin{document}

	\maketitle
	\begin{abstract}
	We propose an adaptive refinement strategy for hierarchical B-splines that increases the approximation power of every marked cell by at least one. Here, the approximation power of a cell is the finest level $\ell$ such that all the B-splines of level $\ell$ that do not vanish on the cell belong to the hierarchical space. This increase can be achieved with the standard refinement of hierarchical meshes, in which each cell is subdivided at most once per step, provided that the regions formed at each level by the cells with this property are nested. We call the hierarchical meshes with this nesting property \emph{weakly admissible}, characterize them, and prove that the proposed refinement preserves weak admissibility. We also prove a complexity estimate: the number of cells created along the adaptive process is bounded by a constant times the total number of marked cells. Moreover, we extend known local $L^q$ estimates for a multilevel quasi-interpolant on these meshes to Sobolev norms, showing that the local error bound on each active cell is governed by its approximation power, which supports using this quantity to drive the refinement. Numerical experiments illustrate the local gain of resolution produced by the proposed strategy, as well as its robustness with respect to the marking criterion and the marking parameter.
	\end{abstract}

\begin{quote}\small

\end{quote}

		\section{Introduction}\label{S:Intro}

		The theory of spline approximation plays a central role in modern numerical analysis. On the one hand, spline spaces allow for the representation of smooth functions using locally supported and stable bases; on the other hand, their approximation properties can be rigorously studied within the framework of Sobolev spaces. In numerous applications---particularly in the formulation of finite element methods---explicit local error estimates are required for linear operators that project or approximate functions in spline spaces. Among these operators, quasi-interpolants are widely recognized as effective and efficient approximation tools~\cite{Schumi, LMS18}. They provide significant advantages by combining simplicity in construction with flexibility in adapting to singular data, making them highly versatile for various computational applications.

		This work focuses on refinement strategies within hierarchical spline spaces and the analysis of approximation errors using quasi-interpolation. Hierarchical B-splines (HB-splines)~\cite{K98, VGJS2011} are defined by a hierarchy of graded meshes that represent different levels of refinement, enabling local approximation.

		In an adaptive procedure, the hierarchical space is enriched by refining the cells where an error indicator is large. However, refining only these cells may not increase sufficiently the resolution of the space on them, and in some cases the refinement of some cells may not even activate any new basis function. One way to avoid this situation is to mark basis functions instead of cells (see, e.g.,~\cite{BG18_aposteriori, GV2018}), since refining the support of an active function always activates new basis functions of the next level.

		To obtain optimal error estimates via quasi-interpolants~\cite{MS15, S17}, existing theory imposes restrictions on the hierarchical mesh. One such condition is the standard notion of admissibility~\cite{BGi15}, which controls, independently of the mesh depth, the number of different levels of truncated hierarchical B-spline (THB-spline) basis functions~\cite{GJS2012} that do not vanish on a given cell. The same concept was introduced in~\cite{GHP2017} for HB-splines, limiting the number of levels to two. A condition that guarantees these properties is the so-called strict admissibility, which allows one to work directly with the mesh structure (hierarchical subdomains) rather than involving the THB-spline or HB-spline functions directly.

		Previous work~\cite{BGV2018} established results relating admissible and strictly admissible meshes of class $m$ for both THB-splines and HB-splines. Furthermore, they presented an adaptive refinement algorithm that takes as input a mesh satisfying strict admissibility of class $m$ alongside a set of marked elements of interest, and outputs a refined mesh preserving this property. Additionally, a refinement complexity result for this algorithm was obtained in~\cite{BGMP2016}, which controls the number of newly created cells relative to the number of marked cells during an iterative procedure.

		Mesh conditions such as admissibility usually lead to the refinement of some additional cells, which often also enriches the space around the marked region. These conditions, however, are mainly motivated by the analysis, since they make it possible to recover in the hierarchical spline setting situations similar to those of standard adaptive finite element methods, where well-established proof techniques are available; the gain of resolution on the marked region is thus a by-product rather than their primary aim. Our approach follows a different path: we aim at increasing the resolution of the hierarchical space on the region where it is required by explicitly selecting the refinement that achieves it.

		This work has a twofold purpose. First, we design a robust adaptive algorithm whose refinement complexity is controlled, aimed at increasing the local approximation power of the hierarchical space on the cells where more resolution is required. An active cell attains the full resolution of its level when all the tensor-product B-splines of its level that do not vanish on it belong to the hierarchical space, that is, when the space locally reproduces the full tensor-product spline space of that level. Our strategy refines enough so that this occurs on the cells where more resolution is required. Once this is achieved on a cell, a higher resolution is obtained by making all its children attain the full resolution of the next level. In order to guarantee that the refinement can be carried out in a single step, that is, every cell is subdivided at most once, the cells of the previous level in a suitable neighborhood of each cell must be either active or already refined. This condition on the cells gives rise to a family of meshes in which the subdomains where each level attains its full resolution are nested. Accordingly, we work with the hierarchical meshes that satisfy this nesting property, and the proposed refinement produces again a mesh of this class. Since strictly admissible meshes of class~$2$ satisfy this property, we call them \textit{weakly admissible}.
		 This class of meshes, although without a specific name at the time, was first considered in~\cite{K98} to prove pointwise approximation properties through the construction of a quasi-interpolant. It was later revisited in~\cite{BG15}, where local approximation properties in $L^q$ were established, which provides a further motivation for working within this setting. As a second objective, this article extends these results to Sobolev norms, showing that this class of meshes provides sufficient control to develop a robust local approximation theory.

		The remainder of this article is organized as follows. Section~\ref{S:WAHM} briefly reviews the notion of hierarchical spline spaces and defines and characterizes the concept of weakly admissible hierarchical meshes. Section~\ref{S:algorithm} presents the design and implementation of a refinement algorithm that increases the approximation power of the hierarchical space on the marked cells and preserves the weak admissibility of the mesh, along with the corresponding complexity result analogous to the one established in~\cite{BGMP2016}. Section~\ref{Sec: local approcimation weakly admissible meshes} establishes optimal local error estimates in higher-order Sobolev norms for a multilevel quasi-interpolant over weakly admissible meshes. Section~\ref{Sec: Test numerico} illustrates the behavior of the proposed refinement strategy through numerical tests, analyzing first a single refinement step for the $L^2$-projection and then the complete adaptive loop for the Poisson equation. Finally, Section~\ref{S:conclusions} summarizes the main conclusions and some directions for future work.

		\section{Weakly admissible meshes for hierarchical splines}\label{S:WAHM}

		Let $\Omega = [0,1]^d$ be a closed hyper-rectangle in $\mathbb{R}^d$. We construct a hierarchy of uniform tensor-product meshes $\{\mathcal{Q}_{\ell}\}_{\ell \in \mathbb{N}_0}$ over $\Omega$ via dyadic refinement. Specifically, a cell $Q \in \mathcal{Q}_{\ell}$ at level $\ell$ is a hypercube $Q = I_1 \times \dots \times I_d$ of side length $h_{\ell} = 2^{-\ell}$. The mesh $\mathcal{Q}_{\ell}$ is obtained by bisecting each edge of the cells in $\mathcal{Q}_{\ell-1}$, yielding the relation $h_{\ell} = h_{\ell-1}/2$. Associated with this nested mesh sequence is a corresponding sequence of nested $d$-variate tensor-product spline spaces $\{\mathcal{S}_{\ell}\}_{\ell \in \mathbb{N}_0}$, satisfying $\mathcal{S}_{\ell} \subset \mathcal{S}_{\ell+1}$ for all $\ell \geq 0$. Each space $\mathcal{S}_{\ell}$ is spanned by a tensor-product B-spline basis $\mathcal{B}_{\ell}$ of maximum smoothness and uniform polynomial degree $\mathbf{p} = (p, \dots, p)$ (cf.~\cite{Schumi}); although different degrees could be considered in each coordinate direction, they are assumed equal throughout to simplify the presentation.

		\begin{definition}[Hierarchy of subdomains]\label{D: hierarchy of subdomains}
			Let $n \in \mathbb{N}$. We say that $\mathbf{\Omega}_{n}=\left\lbrace \Omega_{0}, \Omega_{1},...,\Omega_{n-1},\Omega_n\right\rbrace $ is a hierarchy of subdomains of $\Omega$ of depth $n$ if:
			\begin{itemize}
				\item $\Omega_{\ell}$ is union of cells of level $\ell-1$, for $\ell=1,2,...,n-1.$
				\item $\Omega=\Omega_{0} \supset \Omega_{1} \supset ... \supset \Omega_{n-1} \supset \Omega_{n}= \emptyset$.
			\end{itemize}
		\end{definition}
		Given a hierarchy of subdomains $\mathbf{\Omega}_{n}$, the corresponding hierarchical mesh $\QQQ=\QQQ(\mathbf{\Omega}_{n})$  is given by
		$    \QQQ:=\overset{n-1}{\underset{\ell=0}{\bigcup}} \left\lbrace Q \in \mathcal{Q}_{\ell} \;| \; Q \subset \Omega_{\ell} \wedge Q \not\subset \Omega_{\ell+1} \right\rbrace.$ We say that
		$Q$ is an \emph{active cell} of level $\ell$ if $Q \in \QQQ \cap \mathcal{Q}_{\ell}$. Additionally, we define the set of HB-splines $\mathcal{H}$ as follows
		\begin{equation*}
			\mathcal{H}:=\overset{n-1}{\underset{\ell=0}{\bigcup}} \left\lbrace \beta \in \mathcal{B}_{\ell} \;| \;\supp\beta \subset \Omega_{\ell} \wedge
			\supp\beta \not\subset \Omega_{\ell+1} \right\rbrace,
		\end{equation*}
		where $\supp\beta$ denotes the intersection of the support of $\beta$ with $\Omega_0$. We say that $\beta$ is an \emph{active function} of level $\ell$ if $\beta \in\mathcal{H}\cap\mathcal{B}_{\ell}$.

		\begin{definition}[Hierarchical spline space]
			The space of hierarchical splines in $\Omega$ is given by $\mathrm{span}\mathcal{H}$.
		\end{definition}
		For a more detailed discussion on hierarchical splines we refer to~\cite{GJS14}.

			\smallskip

According to the existing literature, the development of an adaptivity theory for isogeometric methods based on hierarchical splines typically relies on enforcing suitable mesh grading conditions to govern local refinement, which allow the use of proof techniques analogous to those of standard adaptive finite element methods. The concept of admissible meshes~\cite{BGi15} for THB-splines plays a central role in this framework.

\begin{definition}[Admissible hierarchical mesh]\label{Def: Malla Admmisible}
	A mesh $\mathcal{Q}$ is \textit{admissible of class} $m\geq 2$ if the THB-splines taking non-zero values on any cell $Q\in \mathcal{Q}$ belong to at most $m$ successive levels.
\end{definition}

Restricting the hierarchy in this manner limits the overlap of basis functions across different levels. A sufficient condition for admissibility can be established using the following auxiliary domains.

For $\ell=0,1,\dots,n-1, n$ let $\omega_\ell$ be the union of the elements at level $\ell$ whose support extension is fully contained within $\Omega_\ell$; that is,
\begin{equation*}%
	\omega_\ell:=\bigcup_{\substack{Q\in\QQ_\ell\\ \tilde Q\subset\Omega_\ell}} Q,
\end{equation*}
where $\tilde{Q} := \bigcup_{\substack{\beta\in\BB_\ell\\\supp\beta\supset Q}}\supp\beta$ denotes the support extension of $Q$. Equivalently, $\omega_\ell$ is the largest subset of $\Omega_\ell$ such that the basis functions $\{\beta\in\BB_\ell\mid \supp\beta\subset\Omega_{\ell}\}$ span the entire tensor-product space of level $\ell$ restricted to $\omega_\ell$. Noting that $\omega_n = \emptyset$.

\begin{definition}[Strictly admissible hierarchical mesh]\label{Def: Malla Estrictamente Admmisible}
	A mesh $\QQQ$ is \emph{strictly admissible of class} $m\geq 2$ if $\Omega_\ell \subset \omega_{\ell-m+1}$ holds for all $\ell=m, m+1, \dots, n-1$.
\end{definition}

Note that strict admissibility of class $m$ implies admissibility of class $m$ (see, e.g.,~\cite[Proposition 1]{BGV2018}).

The subdomains $\omega_\ell$ also admit a natural interpretation in terms of the local resolution of the hierarchical space. A cell $Q$ of level $\ell$ is contained in $\omega_\ell$ if and only if all the B-splines of level $\ell$ that do not vanish on $Q$ are supported in $\Omega_\ell$, and hence belong to $\Span\HH$; in this sense, on such cells the hierarchical space attains the full resolution of the tensor-product space of level $\ell$. The refinement strategy proposed in \Cref{S:algorithm} increases the resolution on the cells where it is required by refining until they are contained in the subdomain $\omega_\ell$ of their level $\ell$, and, once this holds, by making their children be contained in $\omega_{\ell+1}$.

To include a cell $Q\subset\Omega_\ell$ of level $\ell$ in $\omega_\ell$, some cells of level $\ell-1$ have to be added to $\Omega_\ell$. With the refinement procedure described in \Cref{def: construccion de refinamiento}, this can be done by subdividing active cells of level $\ell-1$ once, into their children of level $\ell$, provided that these cells are contained in $\Omega_{\ell-1}$. As we show in \Cref{Equivalencia_Mallas_DA}, for clustered hierarchies (see \Cref{Def: Jerearquia subdominio agrupada}) this holds for every cell $Q\subset\Omega_\ell$ and every level $\ell$ if and only if the subdomains $\omega_\ell$ are nested. This is precisely the property that defines the following class of meshes.

\begin{definition}[Weakly admissible hierarchical mesh]\label{Def: Debilmente Admisible}
	A mesh $\QQQ$ is a \emph{weakly admissible hierarchical mesh} (WAHM) if $\omega_\ell \subset \omega_{\ell-1}$ holds for all $\ell=1, \dots, n-1$.
\end{definition}

Meshes of this type were initially introduced in~\cite{K98}—albeit without this specific nomenclature—to establish $L^\infty$ approximation properties for hierarchical spline spaces via the construction of a quasi-interpolant. They were later revisited in~\cite{BG15} to analyze local $L^q$ approximation properties. Section~\ref{Sec: local approcimation weakly admissible meshes} advances this theoretical framework by deriving new results specifically tailored to weakly admissible meshes.

\begin{remark}\label{Ob: EA2 es DA}
	The name \emph{weakly admissible} simply reflects the fact that any strictly admissible mesh of class $2$ is also weakly admissible, which is straightforward to verify; the converse is generally false (see~\cite[Remark 4.10 and Figure 3]{BG15}).
\end{remark}

The primary objective of this section is to thoroughly explore and characterize the concept of weakly admissible meshes.

\subsection{Some properties of meshes in $\mathbb{R}^d$}\label{subsec: propiedades de mallas infinitas}

In order to obtain a characterization of weakly admissible meshes, this section presents the study of a family of tensor-product meshes, denoted by $\{\QQ_\ell^\infty\}_{\ell\in\mathbb{N}}$. Here, we establish properties involving different neighborhoods for a given cell $Q\in\QQ_\ell^\infty$.

For each level $\ell\in\mathbb{N}$, we consider the mesh $\QQ_\ell^\infty$ as the one whose cells have their vertices in the set $2^{-\ell}\mathbb{Z}^d$, where $d\in\mathbb{N}$. We also assume a polynomial degree vector $\mathbf{p} = (p,\dots,p)$ in each direction and maximum smoothness in the whole domain. Based on the (univariate) cardinal B-spline $M_p$ of degree $p$, the B-spline function associated to the point $\mathbf{k} = (k_1,\dots,k_d) \in \mathbb{Z}^d$ at level $\ell$ is given by \(B^\ell_\mathbf{k}(x) =
B^\ell_{k_1,\dots,k_d}(x_1,\dots,x_d) =
M_{p}(2^\ell x_1-k_1)\dots M_{p}(2^\ell x_d-k_d)\),
with support
\begin{equation}\label{supp de B-splines d dimensional}
	\supp B^\ell_\mathbf{k} =
	[\tfrac{k_1}{2^\ell},\tfrac{k_1+p+1}{2^\ell}]
	\times \dots \times [\tfrac{k_d}{2^\ell},\tfrac{k_d+p+1}{2^\ell}].
\end{equation}
Thus, $\mathcal{B}^\infty_{\mathbf{p},\ell} = \{B^\ell_\mathbf{k} \mid \mathbf{k}\in\mathbb{Z}^d\}$ is the set of all B-spline functions of level $\ell$. Furthermore, $\mathbf{B}^\infty_{\mathbf{p},\ell} = \{\supp \beta \mid \beta\in \mathcal{B}^\infty_{\mathbf{p},\ell}\}$ is the set of the supports of all B-splines at level $\ell$.
\begin{remark}\label{Obs: repreesentacion forma p}
	From \eqref{supp de B-splines d dimensional}, it follows that the elements of $\mathbf{B}^\infty_{\mathbf{p},\ell}$ are $d$-dimensional cubes formed by $(p+1)\times(p+1)\times\dots\times(p+1)$ cells of level $\ell$.
\end{remark}
In addition, the support extension $\widehat{Q}$ of a cell $Q\in\QQ^\infty_{\ell}$ with respect to $\BB^\infty_{\mathbf{p},\ell}$ is defined as
\begin{equation}\label{Eq: soperte extendido infinito de una celda}
	\widehat{Q} =\displaystyle\bigcup_{\substack{\beta\in\BB^\infty_{\mathbf{p},\ell}\\Q\subset\supp\beta}}\supp\beta.
\end{equation}

We now give some key definitions in order to obtain the results in this section.

\begin{definition}[Parent of a cell]
	Let $\ell\in\NN$ and $Q\in\QQ^\infty_\ell$. For $0\le k<\ell$, we define $\parent_k(Q)$ as the unique cell in $\QQ^\infty_{k}$ satisfying $Q\subset\parent_k(Q)$. In the particular case where $k=\ell-1$, we denote $\parent_{\ell-1}(Q) = \parent(Q)$.
\end{definition}

\begin{definition}[Children of a cell]
	Let $\ell\in\NN$ and $Q\in\QQ^\infty_\ell$. For $k>\ell$, we define $\children_k(Q)$ as the set of cells $Q'\in\QQ^\infty_{k}$ satisfying $Q'\subset Q$. In the particular case where $k = \ell+1$, we denote $\children_{\ell+1}(Q) = \children(Q)$.
\end{definition}

\begin{definition}[$\mathbf{p}$-form of level $\ell$]\label{Def: Formas_p_nivel_l}
	Given a subset $\Omega_0$ of $\mathbb{R}^d$ formed by cells of $\QQ_0^\infty$, we call the elements of the following set a \emph{$\mathbf{p}$-form of level $\ell$}:
	\begin{equation*}\label{Eq: Formas p nivel l}
		\mathbf{F}^\infty_{\mathbf{p},\ell} (\Omega_0) = \{C\in \mathbf{B}^\infty_{\mathbf{p},\ell} \mid \exists \:Q\in\QQ^\infty_{\ell+1} \text{ such that } Q\subset\Omega_0, \widehat{Q} \subset C\},
	\end{equation*}
	where $\widehat{Q}$ is given by \eqref{Eq: soperte extendido infinito de una celda}. For simplicity, we denote the set $\mathbf{F}^\infty_{\mathbf{p},\ell} (\Omega_0)$ as $\mathbf{F}^\infty_{\mathbf{p},\ell}$.
\end{definition}

\begin{remark} \label{Formas_p_interpretacion}
	Note that not all supports of the functions in $\mathcal{B}^\infty_{\mathbf{p},\ell}$ constitute a $\mathbf{p}$-form (see Figure \ref{Fig: Supp de B-spline y Formas p}).

	\begin{figure}[ht]
	\centering
	\begin{minipage}{0.4\textwidth}
		\centering
		\begin{tikzpicture}[scale=2]

				\foreach \x in {0,1/10,3/10,5/10,7/10,9/10}{
					\draw[thin,gray!20] (\x,0) -- (\x,1);
					\draw[thin,gray!20] (0,\x) -- (1,\x);
				}
				\fill[gray!50] (3/5,3/5) rectangle (1,1);

				\draw[dashed] (0,0) rectangle (1,1);
				\foreach \x in {1/5, 2/5, 3/5, 4/5} {
					\draw[dashed] (\x,0) -- (\x,1);
				}
				\foreach \y in {1/5, 2/5, 3/5, 4/5} {
					\draw[dashed] (0,\y) -- (1,\y);
				}

				\draw[dashed] (3/5,3/5) rectangle (1,1);
				\draw[thick] (3/5,3/5) -- (3/5,1);
				\draw[thick] (3/5,3/5) -- (1,3/5);

				\draw[thick, red] (0,0) rectangle (4/5,4/5);

				\node at (1-1/10,9/10) {\tiny $\Omega_0$};
		\end{tikzpicture}
			\vspace{0.2em}

	\end{minipage}
		\hspace{1em}
	\begin{minipage}{0.4\textwidth}
			\centering
		\begin{tikzpicture}[scale=2]

				\foreach \x in {0,1/10,3/10,5/10,7/10,9/10}{
					\draw[thin,gray!20] (\x,0) -- (\x,1);
					\draw[thin,gray!20] (0,\x) -- (1,\x);
				}

				\fill[gray!50] (3/5,3/5) rectangle (1,1);

				\fill[yellow!100] (3/5,3/5) rectangle (7/10,7/10);

				\draw[dashed] (0,0) rectangle (1,1);
				\foreach \x in {1/5, 2/5, 3/5, 4/5} {
					\draw[dashed] (\x,0) -- (\x,1);
				}
				\foreach \y in {1/5, 2/5, 3/5, 4/5} {
					\draw[dashed] (0,\y) -- (1,\y);
				}

				\draw (3/5,3/5) rectangle (7/10,7/10);

				\draw[dashed] (3/5,3/5) rectangle (1,1);
				\draw[thick] (3/5,3/5) -- (3/5,1);
				\draw[thick] (3/5,3/5) -- (1,3/5);

				\draw[thick, red] (1/5,1/5) rectangle (1,1);

				\node at (1-1/10,9/10) {\tiny $\Omega_0$};
		\end{tikzpicture}
		\vspace{0.2em}

	\end{minipage}
	\caption{Portion of the domain $\Omega_0$ (gray) and support of bicubic B-splines of level $\ell$ (red). Left: the support is not a $\mathbf{p}$-form. Right: the support is a $\mathbf{p}$-form as it contains the extended support of a cell of level $\ell+1$ (yellow) contained in $\Omega_0$.}\label{Fig: Supp de B-spline y Formas p}
\end{figure}
\end{remark}

In what follows, for the sake of simplicity and better visualization, respectively, the proofs and mesh plots will be presented for the case $d=2$.

\begin{definition}[Neighborhood of a cell] \label{Def: C_Q}
	Let $Q\in\QQ^\infty_\ell$. We define the set $$\mathcal{N}_Q := \bigcup_{\substack{Q^*\in\QQ^\infty_\ell, \\ Q^*\subset \widehat{Q}}} \parent(Q^*).$$
\end{definition}

The following result establishes that $\mathcal{N}_Q$ is the support of a B-spline of level $\ell-1$, which consists of $(p+1)\times(p+1)$ cells of level $\ell-1$; and moreover, if $Q$ is in $\Omega_0$, $\mathcal{N}_Q$ is a $\mathbf{p}$-form of level $\ell-1$.

\begin{proposition}\label{Prop: C_Q es Forma_p}
	Let $Q\in\QQ^\infty_\ell$. Then, $\mathcal{N}_Q\in \mathbf{B}^\infty_{\mathbf{p},\ell-1}$. Furthermore, if $Q \subset \Omega_0$, it holds that $\mathcal{N}_Q\in \mathbf{F}^\infty_{\mathbf{p},\ell-1}$.
\end{proposition}

\begin{proof}
	For $Q\in\QQ^\infty_\ell$, we consider its support extension $\widehat{Q}$, which is a set formed by $(2p+1)\times(2p+1)$ cells of level $\ell$ centered at $Q$. Then, since $\QQ^\infty_\ell$ is obtained from $\QQ^\infty_{\ell-1}$ by dyadic refinement, $\mathcal{N}_Q$ is formed by $(2p+2)\times(2p+2)$ cells of level $\ell$, or equivalently, $(p+1)\times(p+1)$ cells of level $\ell-1$ (see Figure \ref{Fig:CQ}). Therefore, $\mathcal{N}_Q$ is an element of $\mathbf{B}^\infty_{\mathbf{p},\ell-1}$ for which $\widehat{Q}\subset \mathcal{N}_Q$. Furthermore, if $Q\subset\Omega_0$ using the Definition \ref{Def: Formas_p_nivel_l}, it follows that $\mathcal{N}_Q\in \mathbf{F}^\infty_{\mathbf{p},\ell-1}$.
\end{proof}

\begin{figure}[ht]
	\centering

	\begin{minipage}{0.4\textwidth}
		\centering
		\begin{tikzpicture}[scale=2]
			\fill[yellow!100] (0.0,0.2) rectangle (0.8,1.0);
			\draw[very thick] (0,0) rectangle (1.2,1.2);

			\foreach \x in {0.2, 0.4, 0.6, 0.8, 1} {
				\draw (\x,0) -- (\x,1.2);
			}

			\foreach \y in {0.2, 0.4, 0.6, 0.8, 1} {
				\draw (0,\y) -- (1.2,\y);
			}

			\draw[very thick] (0.1,0.2) rectangle (0.8,0.9);

			\foreach \x in {0.3, 0.5, 0.7} {
				\draw (\x,0.2) -- (\x,0.9);
			}

			\foreach \y in {0.3, 0.5, 0.7} {
				\draw (0.1,\y) -- (0.8,\y);
			}

			\draw (0.5,0.5) -- (0.5,0.6);
			\draw (0.4,0.5) -- (0.5,0.5);

			\node at (0.45,0.55) {\tiny $Q$};

		\end{tikzpicture}
		\vspace{0.2em}

		{\scriptsize $\mathbf{p}=(3,3)$}
	\end{minipage}
	\hspace{2em}
	\begin{minipage}{0.4\textwidth}
		\centering
		\begin{tikzpicture}[scale=2]

			\fill[yellow!100] (0.0,0.0) rectangle (1.0,1.0);

			\draw[very thick] (0,0) rectangle (1.2,1.2);

			\foreach \x in {0.2, 0.4, 0.6, 0.8, 1} {
				\draw (\x,0) -- (\x,1.2);
			}

			\foreach \y in {0.2, 0.4, 0.6, 0.8, 1} {
				\draw (0,\y) -- (1.2,\y);
			}

			\draw[very thick] (0.,0.1) rectangle (0.9,1.);

			\foreach \x in {0.1, 0.3, 0.5, 0.7} {
				\draw (\x,0.1) -- (\x,1.);
			}

			\foreach \y in {0.3, 0.5, 0.7, 0.9} {
				\draw (0.,\y) -- (0.9,\y);
			}

			\draw (0.5,0.5) -- (0.5,0.6);
			\draw (0.4,0.5) -- (0.5,0.5);

			\node at (0.45,0.55) {\tiny $Q$};

		\end{tikzpicture}
		\vspace{0.2em}

		{\scriptsize $\mathbf{p}=(4,4)$}
	\end{minipage}
	\caption{Neighborhood $\mathcal{N}_Q$ of a cell $Q$, distinguishing the refined region $\widehat{Q}$ from the highlighted area $\mathcal{N}_Q$.}\label{Fig:CQ}
\end{figure}

The set $\mathcal{N}_Q$ for a cell $Q\in\QQ_\ell$ was previously considered in \cite[Definition 5]{BGV2018}. This neighborhood of $Q$ will be highly useful in the following section, and its properties are fundamental for characterizing weakly admissible meshes.

\begin{proposition} \label{Prop: representacion de supp B-splines}
	If $C \in \mathbf{B}^\infty_{\mathbf{p},\ell-1}$, then there exist unique cells $\core(C) := \{Q_i\}_{i=1}^{2^d} \subset \QQ^\infty_\ell$ such that:
	\begin{enumerate}[label=(\roman*)]
		\item \label{item1} $C = \mathcal{N}_{Q_i}$ for $i=1,\dots,2^d$.
		\item \label{item2} $C = \bigcup_{i=1}^{2^d} \widehat{Q_i}$.
		\item \label{item3} Additionally, if $C\in \mathbf{F}^\infty_{\mathbf{p},\ell-1}$, then
		\begin{equation*}\label{Eq: representación de forma p}
			C \: \cap \: \Omega_0 = \bigcup_{\substack{Q_i\in \core(C), \\ Q_i\subset\Omega_0}} \: \widehat{Q_i} \:\cap\: \Omega_0.
		\end{equation*}
	\end{enumerate}
\end{proposition}

\begin{proof}
	We provide the proof for $d=2$, the general case follows from the same argument.
	We prove \ref{item1} and \ref{item2} simultaneously. We consider $C \in \mathbf{B}^\infty_{\mathbf{p},\ell-1}$, which is a set formed by $(p+1)\times(p+1)$ cells of level $\ell-1$ (see \Cref{Obs: repreesentacion forma p}). Then, through dyadic refinement, $C$ is a set formed by $(2p+2)\times(2p+2)$, or equivalently, $(p+2+p)\times(p+2+p)$ cells of level $\ell$ (see Figure \ref{Fig: C como union de extensiones}). This indicates that there are four unique cells $Q_1, Q_2, Q_3, Q_4 \in \QQ^\infty_\ell$ (since $2^d=4$ when $d=2$) such that $C = \mathcal{N}_{Q_i}$ for $i = 1,2,3,4$. Furthermore, $C = \bigcup_{i=1}^4 \widehat{Q_i}$.

	\begin{figure}[ht]
		\centering
		\begin{minipage}{0.4\textwidth}
			\centering
		\begin{tikzpicture}[scale=2]
			\draw[very thick] (0,0) rectangle (1,1);

			\foreach \x in {0.2, 0.4, 0.6, 0.8} {
				\draw (\x,0) -- (\x,1);
			}

			\foreach \y in {0.2, 0.4, 0.6, 0.8} {
				\draw (0,\y) -- (1,\y);
			}

			\node at (0.5,1.1) {\scriptsize $p+1$};
		\end{tikzpicture}
		\vspace{0.2em}

	\end{minipage}
	\hspace{2em}
	\begin{minipage}{0.4\textwidth}
		\centering
		\begin{tikzpicture}[scale=2]

			\fill[yellow!90] (0.4,0.4) rectangle (0.6,0.6);
			\draw[very thick] (0,0) rectangle (1,1);

			\foreach \x in {0.1, 0.2, 0.3, 0.4, 0.5, 0.6, 0.7, 0.8, 0.9} {
				\draw (\x,0) -- (\x,1);
			}

			\foreach \y in {0.1, 0.2, 0.3, 0.4, 0.5, 0.6, 0.7, 0.8, 0.9} {
				\draw (0,\y) -- (1,\y);
			}

			\node at (0.5,1.1) {\scriptsize $p+2+p$};
		\end{tikzpicture}
		\vspace{0.2em}

	\end{minipage}

	\caption{Dyadically refining a set $C \in \mathbf{B}^\infty_{\mathbf{p},\ell-1}$ (on the left) we obtain $4$ cells $Q$ of level $\ell$ (highlighted on the right) that have the same $\mathcal{N}_Q$, which coincides with $C$, and moreover, the union of their support extensions equals $C$.} \label{Fig: C como union de extensiones}
\end{figure}

To prove \ref{item3}, suppose that $C\in \mathbf{F}^\infty_{\mathbf{p},\ell-1}$. If $Q_1, Q_2, Q_3, Q_4 \subset \Omega_0$, the result follows from~\ref{item2}. Otherwise, we can write:
\[
C \: \cap \: \Omega_0 = \left[ \bigcup_{\substack{Q_i\in \core(C), \\ Q_i\subset\Omega_0}} \: \widehat{Q_i} \: \cap \: \Omega_0 \right] \: \cup \left[ \bigcup_{\substack{Q_i\in \core(C), \\ Q_i\not\subset\Omega_0}} \: \widehat{Q_i} \: \cap \: \Omega_0 \right] =: A \cup B.
\]
It is sufficient to show that $B \subset A$. This situation can only occur in the distinct cases, which are illustrated in Figure \ref{Fig: tres casos}, where the inclusion holds.

\begin{figure}[ht]
	\centering

	\begin{minipage}{0.2\textwidth}
		\centering
		\begin{tikzpicture}[scale=.7]

			\fill[gray! 25] (0.5,-1) rectangle (2,2);

			\draw (0.5,-1) -- (0.5,2);
			\draw (-1,-1) rectangle (2,2);
			\draw (0,0) rectangle (1,1);
			\foreach \x in {0.5} {
				\draw (\x,0) -- (\x,1);
			}
			\foreach \y in {0.5} {
				\draw (0,\y) -- (1,\y);
			}

			\node at (1.5,0.5) {\scriptsize $\Omega_0$};

			\node at (0.5,2.3) {\scriptsize $\mathcal{N}_{Q_i}$};
		\end{tikzpicture}
		\vspace{0.2em}

	\end{minipage}
	\hspace{1em}
	\begin{minipage}{0.3\textwidth}
		\centering
		\begin{tikzpicture}[scale=.7]

			\fill[gray! 25] (-1,0.5) rectangle (2,2);

			\draw (-1,0.5) -- (2,0.5);

			\draw (-1,-1) rectangle (2,2);
			\draw (0,0) rectangle (1,1);
			\foreach \x in {0.5} {
				\draw (\x,0) -- (\x,1);
			}
			\foreach \y in {0.5} {
				\draw (0,\y) -- (1,\y);
			}

			\node at (0.5,1.4) {\scriptsize $\Omega_0$};						\node at (0.5,2.3) {\scriptsize $\mathcal{N}_{Q_i}$};
		\end{tikzpicture}
		\vspace{0.2em}

	\end{minipage}
	\hspace{1em}
	\begin{minipage}{0.2\textwidth}
		\centering
		\begin{tikzpicture}[scale=.7]

			\fill[gray! 25] (0.5,0.5) rectangle (2,2);

			\draw (0.5,0) -- (0.5,2);
			\draw (0,0.5) -- (2,0.5);
			\draw (-1,-1) rectangle (2,2);

			\draw (0,0) rectangle (1,1);
			\foreach \x in {0.5} {
				\draw (\x,0) -- (\x,1);
			}
			\foreach \y in {0.5} {
				\draw (0,\y) -- (1,\y);
			}

			\node at (1.5,1.5) {\scriptsize $\Omega_0$};						\node at (0.5,2.3) {\scriptsize $\mathcal{N}_{Q_i}$};
		\end{tikzpicture}
		\vspace{0.2em}

	\end{minipage}

	\caption{The possible configurations for the cells $\{Q_i\}_{i=1}^4 \subset \QQ^\infty_\ell$ that satisfy \ref{item2} in Proposition \ref{Prop: representacion de supp B-splines} for which some $Q_i$ is not contained in $\Omega_0=[0,1]^d$.} \label{Fig: tres casos}
\end{figure}
\end{proof}

\begin{proposition} \label{C_Q*_incluido_extendido_parent_Q}
If $Q\in\QQ^\infty_\ell$, then $\widehat{\parent(Q)} = \bigcup\{\mathcal{N}_{Q^*} \,\mid\,Q^*\in\QQ^\infty_\ell,  Q^*\subset \widehat{Q}\}$.
\end{proposition}

\begin{proof}
Let $Q\in\QQ^\infty_\ell$. First, note that $\widehat{\parent(Q)}$ is the unique set formed by $(2p+1)\times(2p+1)$ cells of level $\ell-1$ centered at the cell $\parent(Q)$. On the other hand, by considering $\widehat{Q}$, let be the set
\(
A := \bigcup\{ \widehat{Q^*}\,\mid\,Q^*\in\QQ^\infty_\ell, Q^*\subset \widehat{Q}\},
\)
consisting of $(4p+1)\times(4p+1)$ cells of level $\ell$ centered at $Q$ (see Figure \ref{Fig: extension del padre de Q}). Since \(A=\bigcup\{Q^{**}\,\mid\,Q^{**}\in\QQ^\infty_\ell, Q^{**}\subset \widehat{Q^*}, Q^*\in\QQ^\infty_\ell, Q^*\subset \widehat{Q}\}\)
and taking Definition \ref{Def: C_Q} into account, we obtain that \[\bigcup\{\parent(Q^{**})\,\mid\,Q^{**}\in\QQ^\infty_\ell,\, Q^{**}\subset A\} = \bigcup\{\mathcal{N}_{Q^*} \,\mid\,Q^*\in\QQ^\infty_\ell,  Q^*\subset \widehat{Q}\}.\]
\begin{figure}[ht]
	\centering

	\begin{minipage}{0.4\textwidth}
		\centering
		\begin{tikzpicture}[scale=0.25]

			\fill[yellow!] (4,4) rectangle (5,5);

			\foreach \x in {0,1,...,10} {
				\draw[thin,gray!20] (\x,0) -- (\x,10);
				\draw[thin,gray!20] (0,\x) -- (10,\x);
			}

			\foreach \x in {0,2,...,10} {
				\draw (\x,0) -- (\x,10);
				\draw (0,\x) -- (10,\x);
			}

			\foreach \i in {5} {
				\node at (\i,11) {$2p+1$};
			}

			\draw[very thick] (4,4) rectangle (6,6);

		\end{tikzpicture}
	\end{minipage}
	\hfill
	\begin{minipage}{0.4\textwidth}
		\centering
		\begin{tikzpicture}[scale=0.25]

			\fill[gray!50] (0,0) rectangle (9,9);

			\fill[gray!25] (2,2) rectangle (7,7);

			\fill[yellow!] (4,4) rectangle (5,5);

			\foreach \x in {0,1,...,10} {
				\draw[thin,gray!80] (\x,0) -- (\x,10);
				\draw[thin,gray!80] (0,\x) -- (10,\x);
			}

			\foreach \x in {0,2,...,10} {
				\draw (\x,0) -- (\x,10);
				\draw (0,\x) -- (10,\x);
			}

			\foreach \i in {5} {
				\node at (\i,11) {$4p+2$};
			}

		\end{tikzpicture}
	\end{minipage}
	\caption{%
	The set $\widehat{\parent{Q}}$ is represented as the union of cells of level $\ell-1$ (left) and level $\ell$ (right), where $Q$ denotes the cell of level $\ell$ shown in yellow. In the right figure, the set $\widehat{Q}$ is depicted in light gray, while the set $A$, defined as the union of support extensions of cells of level $\ell$ in $\widehat{Q}$, is shown in dark gray.}\label{Fig: extension del padre de Q}
\end{figure}
We complete the proof by noticing that this last set consists of $(4p+2)\times(4p+2)$ cells of level $\ell$, or equivalently, $(2p+1)\times(2p+1)$ cells of level $\ell-1$, centered at $\parent(Q)$.
\end{proof}

\begin{remark}\label{parent_extendido_Q_incluido_extendido_Parent_Q}
Notice that, for each $Q\in\QQ^\infty_\ell$, the set $\mathcal{N}_Q$ represents a neighborhood of $Q$ that contains the support extension of $Q$ and is contained within the support extension of the parent of $Q$, i.e.,
\[
\widehat{Q}\subset \mathcal{N}_Q \subset \widehat{\parent(Q)}.
\]
Indeed, the first inclusion follows from \Cref{Def: C_Q}, while the second follows from the previous proposition.
\end{remark}

\subsection{Characterization of weakly admissible hierarchical meshes}\label{subsec: Caracterizacion MDA}

The objective of this section is to present two characterizations of weakly admissible hierarchical meshes (WAHMs) based on the cell neighborhoods $\mathcal{N}_Q$. The first characterization relates cells of level $\ell$ within $\omega_\ell$ to the set $\Omega_{\ell-1}$, while the second relates those within $\Omega_\ell$ to $\Omega_{\ell-1}$.

We consider $\Omega= [0,1]^d \subset \mathbb{R}^d$ and a hierarchical spline space constructed from a hierarchy of subdomains of $\Omega$ with depth $n$, as explained at the beginning of~\Cref{S:WAHM}. We first observe the relationship between the support extensions $\widehat{Q}$ and $\tilde{Q}$, which correspond to the infinite meshes $\{\QQ^\infty_\ell\}_{\ell\in\mathbb{Z}}$ and the bounded meshes $\{\QQ_\ell\}_{\ell\in\mathbb{Z}}$, respectively.

\begin{remark}\label{Obs: extension infinita y extension}
	For any cell $Q\in \QQ^\infty_\ell$ such that $Q\subset \Omega_0$, it holds that
	\(
	\tilde{Q} = \widehat{Q} \cap \Omega_0.
	\)
	This follows directly from the one-to-one correspondence between the B-splines $\beta \in \BB_\ell$ and the B-splines $\beta' \in \BB^\infty_{\mathbf{p},\ell}$. Since their supports satisfy $\supp\beta = \supp\beta' \cap \Omega_0$, we obtain:
	\[
	\widehat{Q} \cap \Omega_0 = \bigcup_{\substack{\beta'\in\BB^\infty_{\mathbf{p},\ell}\\\supp\beta'\supset Q}} (\supp\beta' \cap \Omega_0) = \bigcup_{\substack{\beta\in\BB_\ell\\\supp\beta\supset Q}} \supp\beta = \tilde{Q}.
	\]
\end{remark}

\begin{proposition} \label{extendido_parent_Q_incluido_Omega_ell-1}
	If $\QQQ$ is a WAHM, the support extension of the parents of all cells of level $\ell$ contained in $\omega_\ell$ is contained in $\Omega_{\ell-1}$, i.e.,
	\begin{equation*}
		\bigcup_{Q\in\QQ_\ell, Q\subset\omega_\ell} \widetilde{\parent(Q)} \subset \Omega_{\ell-1}.
	\end{equation*}
\end{proposition}

\begin{proof}
	Let $Q\in\QQ_\ell$ such that $Q\subset\omega_\ell$. By the weak admissibility of the mesh, it follows that $Q\subset\omega_{\ell-1}$. Since $\omega_{\ell-1}$ is a union of cells of level $\ell-1$, we necessarily have $\parent(Q)\subset\omega_{\ell-1}$, from which the result follows.
\end{proof}

The following theorem characterizes the cells $Q\subset\omega_\ell$ in terms of the neighborhood $\mathcal{N}_Q$.

\begin{theorem}[Characterization of $\omega_\ell$] \label{Teo: Omega_ell y omega_ell}
	Let $\QQQ$ be a hierarchical mesh. For any cell $Q\in\QQ_\ell$, it holds that $Q\subset\omega_\ell$ if and only if $\mathcal{N}_Q \cap \Omega_0 \subset \Omega_{\ell}$.
\end{theorem}

\begin{proof}
	Let $Q\in\QQ_\ell$ and assume that $\mathcal{N}_Q \cap \Omega_0 \subset \Omega_\ell$. Since $\widehat{Q} \subset \mathcal{N}_Q$, applying \Cref{Obs: extension infinita y extension} yields
	\(
	\tilde{Q} = \widehat{Q} \cap \Omega_0 \subset \mathcal{N}_Q \cap \Omega_0 \subset \Omega_\ell,
	\)
	which implies $Q \subset \omega_\ell$. Conversely, let $Q\in\QQ_\ell$ such that $Q\subset\omega_\ell$. Then $\widehat{Q} \cap \Omega_0= \tilde{Q}  \subset \Omega_\ell$. Since $\Omega_\ell$ consists of cells of level $\ell-1$, it follows that $\parent(Q^*) \subset \Omega_\ell$ for every $Q^* \in \QQ^\infty_\ell$ such that $Q^* \subset \widehat{Q} \cap \Omega_0$. Consequently, $\mathcal{N}_Q \cap \Omega_0 \subset \Omega_\ell$.
\end{proof}

We now establish the first characterization of weakly admissible hierarchical meshes.

\begin{theorem}[First WAHM characterization] \label{Teo: 1ra equivalencia MDA}
	A hierarchical mesh $\QQQ$ is a WAHM if and only if, for every cell $Q \in \QQ_\ell$ such that $Q \subset \omega_\ell$, it holds that $\mathcal{N}_{\parent(Q)} \cap \Omega_0 \subset \Omega_{\ell-1}$.
\end{theorem}

\begin{proof}
	Let $Q \in \QQ_\ell$ such that $Q \subset \omega_\ell$. If $\QQQ$ is a WAHM, then $\parent(Q) \subset \omega_{\ell-1}$. By \Cref{Teo: Omega_ell y omega_ell}, this is equivalent to $\mathcal{N}_{\parent(Q)} \cap \Omega_0 \subset \Omega_{\ell-1}$.
	Conversely, suppose $\QQQ$ is not a WAHM. Then, for some level $\ell$, there exists $Q \in \QQ_\ell$ such that $Q \subset \omega_\ell$ but $Q \not\subset \omega_{\ell-1}$. Since $\parent(Q) \in \QQ_{\ell-1}$ and $\parent(Q) \not\subset \omega_{\ell-1}$, \Cref{Teo: Omega_ell y omega_ell} implies $\mathcal{N}_{\parent(Q)} \cap \Omega_0 \not\subset \Omega_{\ell-1}$.
\end{proof}

The second characterization relates the cells of $\Omega_\ell$ to the subdomain $\Omega_{\ell-1}$. In general, the condition $\mathcal{N}_Q \cap \Omega_0 \subset \Omega_{\ell-1}$ for all $Q \subset \Omega_\ell$ is not always satisfied by WAHMs, as it requires $\Omega_\ell$ to be free of \emph{isolated} deactivated cells (see \Cref{Fig: contraejemplo_malla_DA}, left). To address this, we restrict our focus to hierarchies that avoid such configurations.

\begin{figure}[H]
	\centering

	\begin{minipage}{0.32\textwidth}
		\centering
		\begin{tikzpicture}[scale=1.7]

			\fill[gray!50] (0.5,0.5) rectangle (0.7,0.7);

			\fill[gray!] (0.55,0.55) rectangle (0.65,0.65);

			\draw[thick] (0,0) rectangle (1.2,1.2);
			\foreach \x in {0.2, 0.4, 0.6, 0.8, 1} {
				\draw[thin] (\x,0) -- (\x,1.2);
			}
			\foreach \y in {0.2, 0.4, 0.6, 0.8, 1} {
				\draw[thin] (0,\y) -- (1.2,\y);
			}

			\foreach \x in {0.3, 0.5, 0.7, 0.9} {
				\draw[thin] (\x,0.2) -- (\x,1);
			}
			\foreach \y in {0.3, 0.5, 0.7, 0.9} {
				\draw[thin] (0.2,\y) -- (1,\y);
			}

			\foreach \x in {0.45, 0.55, 0.65, 0.75} {
				\draw[thin] (\x,0.4) -- (\x,0.8);
			}
			\foreach \y in {0.45, 0.55, 0.65, 0.75} {
				\draw[thin] (0.4,\y) -- (0.8,\y);
			}
			\draw[thin] (0.25,0.2) -- (0.25,0.3);
			\draw[thin] (0.2,0.25) -- (0.3,0.25);
		\end{tikzpicture}
	\end{minipage}
	\hfill
	\begin{minipage}{0.32\textwidth}
		\centering
		\begin{tikzpicture}[scale=2]
			\foreach \x in {0,1/8,2/8,3/8,4/8,5/8,6/8,7/8,1}{
				\draw[thin,gray!20] (\x,0) -- (\x,1);
				\draw[thin,gray!20] (0,\x) -- (1,\x);
			}

			\draw (1/4,1/4) rectangle (1,1);

			\foreach \x in {1/4, 2/4, 3/4} {
				\draw (\x,1/4) -- (\x,1);
				\draw (1/4,\x) -- (1,\x);
			}

			\foreach \x in {3/8, 5/8} {
				\draw (\x,1/4) -- (\x,3/4);
				\draw (1/4,\x) -- (3/4,\x);
			}

			\draw[red!] (0,0) rectangle (3/4,3/4);
			\node at (1.2,0.8) {\footnotesize $\Omega_0$};
			\node at (1.2,0.5) {\footnotesize $\Omega_1$};

		\end{tikzpicture}
	\end{minipage}
	\hfill
	\begin{minipage}{0.32\textwidth}
		\centering
		\begin{tikzpicture}[scale=2]

			\foreach \x in {0,1/8,2/8,3/8,4/8,5/8,6/8,7/8,1}{
				\draw[thin,gray!20] (\x,0) -- (\x,1);
				\draw[thin,gray!20] (0,\x) -- (1,\x);
			}

			\draw (1/4,1/4) rectangle (1,1);

			\foreach \x in {1/4, 2/4, 3/4} {
				\draw (\x,1/4) -- (\x,1);
				\draw (1/4,\x) -- (1,\x);
			}

			\foreach \x in {3/8, 5/8} {
				\draw (\x,1/4) -- (\x,2/4);
			}
			\foreach \x in {3/8} {
				\draw (1/4,\x) -- (3/4,\x);
			}

			\node at (1.2,0.8) {\footnotesize $\Omega_0$};
			\node at (1.2,0.4) {\footnotesize $\Omega_1$};
		\end{tikzpicture}
	\end{minipage}
	\caption{Left: A three-level WAHM with $\mathbf{p}=(3,3)$ containing an isolated deactivated cell of level $1$. Middle and right: Meshes with $\mathbf{p}=(2,2)$. The middle mesh arises from a clustered hierarchy because $\Omega_1$ is associated to a $\mathbf{p}$-form, whereas in the right mesh, $\Omega_1$ is merely a union of cells of level $0$.} \label{Fig: contraejemplo_malla_DA}
\end{figure}

Isolated areas inside $\Omega_\ell$ that do not cover the support of at least one B-spline of level $\ell-1$ result in a hierarchical space lacking active functions at level $\ell$; this is an instance of the refinements that do not increase the resolution of the space mentioned in Section~\ref{S:Intro}. For practical purposes, such cells can be excluded from $\Omega_\ell$. While more general hierarchies are discussed in \cite{VGJS2011, GJS2012}, Kraft \cite{K98} restricts subdomains to unions of B-spline supports. This ensures the space maintains representation power consistent with its refinement level.

\begin{definition}[Clustered hierarchy of subdomains] \label{Def: Jerearquia subdominio agrupada}
	A hierarchy $\mathbf{\Omega}_n$ is \emph{clustered} if, for every $\ell$, there exists an index set $J$ such that
	\(
	\Omega_{\ell} = \bigcup_{i \in J} (C_i \cap \Omega_0),
	\)
	where $C_i \in \mathbf{F}^\infty_{\mathbf{p},\ell-1}$ are ${\bf p}$-forms of level $\ell-1$.
\end{definition}

This definition ensures that $\Omega_\ell$ is a union of supports of B-splines of level $\ell-1$, localized to $\Omega_0$ (see \Cref{Fig: contraejemplo_malla_DA}, middle). Under this assumption, we can derive alternative representations for $\Omega_\ell$.

\begin{theorem}\label{Teo: jerarquia de subdominio agrupada}
	In a clustered hierarchy $\mathbf{\Omega}_n$, the following identities hold for every $\ell=1,\dots,n-1$:
	\begin{equation}\label{Eq: Omega_ell_desde_omega_ell_1}
		\Omega_\ell = \bigcup_{Q \in \QQ_\ell, Q \subset \omega_\ell} (\mathcal{N}_{Q} \cap \Omega_0) = \bigcup_{Q \in \QQ_\ell, Q \subset \omega_\ell} \tilde{Q}.
	\end{equation}
\end{theorem}

\begin{proof}
	For any $Q \in \QQ_\ell$, we have $\tilde{Q} \subset \mathcal{N}_Q \cap \Omega_0$. If $Q \subset \omega_\ell$, then $\mathcal{N}_Q \cap \Omega_0 \subset \Omega_\ell$ by \Cref{Teo: Omega_ell y omega_ell}, which confirms the forward inclusions.
	Conversely, let $C_i \cap \Omega_0 \subset \Omega_\ell$ for some $i \in J$. By \Cref{Prop: representacion de supp B-splines}, $C_i = \mathcal{N}_{Q_k^i}$ for $Q_k^i \in \core(C_i)$. Since $\mathcal{N}_{Q_k^i} \cap \Omega_0 \subset \Omega_\ell$, \Cref{Teo: Omega_ell y omega_ell} implies $Q_k^i \subset \omega_\ell$. The reverse inclusions then follow from the fact that $C_i \cap \Omega_0 = \bigcup \tilde{Q}_k^i$.
\end{proof}

By \Cref{Teo: Omega_ell y omega_ell}, a cell $Q$ of level $\ell$ is contained in $\omega_\ell$ if and only if $\mathcal{N}_Q \cap \Omega_0 \subset \Omega_\ell$. The following result shows that, in a WAHM associated with a clustered hierarchy, the set $\mathcal{N}_Q \cap \Omega_0$ is contained in $\Omega_{\ell-1}$ for every cell $Q \subset \Omega_\ell$ of level $\ell$. Therefore, each cell of level $\ell-1$ contained in $\mathcal{N}_Q \cap \Omega_0$ is either already contained in $\Omega_\ell$ or an active cell of the mesh. This property will be essential in \Cref{S:algorithm} to prove that the proposed refinement strategy is well defined (see \Cref{Lemma: buena definicion de la malla}), since it allows the required refinement to be performed in a single step, in which each active cell is refined at most once.

\begin{theorem}[Second WAHM characterization] \label{Equivalencia_Mallas_DA}
	Let $\QQQ$ be constructed from a clustered hierarchy. Then $\QQQ$ is a WAHM if and only if, for every cell $Q \in \QQ_\ell$ such that $Q \subset \Omega_\ell$, it holds that $\mathcal{N}_Q \cap \Omega_0 \subset \Omega_{\ell-1}$.
\end{theorem}

\begin{proof}
	Suppose $\QQQ$ is a WAHM and let $Q^* \in \QQ_\ell$ such that $Q^* \subset \Omega_\ell$. By \eqref{Eq: Omega_ell_desde_omega_ell_1}, there exists $Q \in \QQ_\ell$ with $Q \subset \omega_\ell$ such that $Q^* \subset \tilde{Q}$. Taking into account Proposition \ref{C_Q*_incluido_extendido_parent_Q} and Remark \ref{Obs: extension infinita y extension}, we obtain that
	\begin{equation} \label{Eq: extension padre de Q local}
		\widetilde{\parent(Q)} = \bigcup_{\substack{Q'\in\QQ_\ell,\\ Q'\subset \tilde{Q}}}( \mathcal{N}_{Q'} \cap \Omega_0).
	\end{equation}
Using the last identity, it follows that $\mathcal{N}_{Q^*} \cap \Omega_0 \subset \widetilde{\parent(Q)}$. Since the mesh is a WAHM, $\widetilde{\parent(Q)} \subset \Omega_{\ell-1}$, and thus $\mathcal{N}_{Q^*} \cap \Omega_0 \subset \Omega_{\ell-1}$.
	The converse follows by contrapositive: if $\QQQ$ is not a WAHM, there exists $Q \in \QQ_\ell$ such that $Q \subset \omega_\ell$ but $\widetilde{\parent(Q)} \not\subset \Omega_{\ell-1}$. The identity~\eqref{Eq: extension padre de Q local} then implies that for at least one $Q^* \subset \tilde{Q} \subset \Omega_\ell$, we must have $\mathcal{N}_{Q^*} \cap \Omega_0 \not\subset \Omega_{\ell-1}$.
\end{proof}

			\section{A new adaptive refinement strategy for HB-splines}\label{S:algorithm}

			In this section, we present an adaptive refinement algorithm applied to a weakly admissible mesh (WAHM) constructed from a clustered hierarchy of subdomains. Given a set of marked cells, the algorithm selects the cells that must be refined so that the approximation power of the hierarchical space increases on the marked cells, and completes this selection, if necessary, so that the refined mesh is again weakly admissible. The weak admissibility of the current mesh guarantees that all the cells to be refined are active cells of the current mesh, so that the refinement is performed in a single step according to \Cref{def: construccion de refinamiento}, in which each active cell is refined at most once (see \Cref{Lemma: buena definicion de la malla}, whose proof relies on \Cref{Equivalencia_Mallas_DA}). The resulting procedure can be implemented recursively, following the ideas of the algorithm presented in~\cite{BGV2018} for strictly admissible meshes of class $m$.
			 First, we recall the standard refinement procedure for hierarchical splines using a enlargement of the hierarchy of subdomains.

			\begin{definition}[Refinement of a hierarchical mesh]\label{def: construccion de refinamiento}
				Let $\mathbf{\Omega}_n:=\{\Omega_0,\dots,\Omega_n\}$ be a hierarchy of subdomains and $\QQQ \equiv \QQQ({\bf\Omega}_n)$ the associated hierarchical mesh. If $\MMM = \{\MM_\ell\}_{\ell=0}^{n-1}$, where $\mathcal{M}_\ell\subset \QQQ\cap\QQ_\ell$ is a set of \emph{marked cells of level $\ell$}, for $\ell=0,\dots,n-1$, we define the refined mesh $\QQQ^* \equiv \QQQ^*({\bf\Omega}_{n+1}^*)$ associated to the enlargement $\mathbf{\Omega}_{n+1}^*:=\{\Omega^*_0,\dots,\Omega^*_{n+1}\}$ given by
				\begin{equation*} \label{Eq: refinamiento por construccion}
					\begin{cases}
						\Omega^*_0:=\Omega_0, \\
						\Omega^*_\ell:=\Omega_\ell \cup \{Q \: |\: Q\in \mathcal{M}_{\ell-1}\}, \quad \ell=1,\dots, n,\\
						\Omega^*_{n+1}= \emptyset.
					\end{cases}
				\end{equation*}
			\end{definition}
		Notice that $\Omega^*_\ell$ is obtained from $\Omega_\ell$ by adding the marked cells of level $\ell-1$. We also remark that the hierarchical spaces generated by the hierarchical bases $\mathcal{H}$ and $\mathcal{H}^*$, associated with the hierarchies of subdomains $\mathbf{\Omega}_n$ and $\mathbf{\Omega}^*_{n^*}$, respectively, satisfy \(\Span\mathcal{H}\subset\Span\mathcal{H}^*\)~(cf.~\cite[Proposition 4]{VGJS2011}).

			Let us consider a weakly admissible hierarchical mesh $\QQQ$, where $\mathbf{\Omega}_n = \{\Omega_\ell\}_{\ell=0}^n$ is a clustered hierarchy of subdomains, and let $\MMM = \{\MM_\ell\}_{\ell=0}^{n-1}$ be a subset of active cells in $\QQQ$ given by
			\(
			\MM_\ell\subset\QQ_{\ell}\cap\QQQ.
			\)
			Prior to the standard refinement, we reorder and swap (if necessary) the cells in  $\MMM$ according to a specific property of the cells and the mesh $\QQQ$. This characteristic allows us to quantify the local approximation power of the hierarchical space. We formalize this property in the following definition and remark before detailing the full procedure in~\Cref{alg:update_marked_elements}.

			\begin{definition}[Approximation power of a cell]\label{Def: poder de aproximacion}
				Let $Q\in\QQ_{\ell}$ such that $Q\subset\Omega_\ell$, for $\ell=0,1,\dots,n-1$. We say that $Q$ \emph{has approximation power $k$ in $\QQQ$}, and we write $Q\in\AAA_k(\QQQ)$, if $k$ is the largest index such that
				\(
				Q\subset\omega_k
				\)
				and
				\(
				Q\not\subset\omega_{k+1}.
				\)
				We say that a cell $Q\in\QQ_\ell$ is optimal or has optimal approximation power if $Q\in\AAA_\ell(\QQQ)$, and we say it is suboptimal or has suboptimal approximation power if $Q\in\AAA_{\ell-1}(\QQQ)$.
			\end{definition}

			\begin{remark}\label{obs: poder aproximacion activas}
				Let $Q\in\QQ_{\ell}$ be an active cell, then $Q\in\AAA_k(\QQQ)$ for some $k\leq \ell$. Indeed, since $Q$ is an active cell of level $\ell$, we have that $Q \not\subset \Omega_j$, and consequently, $Q \not\subset \omega_j$ for all $j > \ell$. Therefore, the approximation power of $Q$ is at most $\ell$.
			\end{remark}

			\begin{algorithm}[H]
				\small
				\caption{\textbf{update\_marked\_elements}}\label{alg:update_marked_elements}
				\begin{algorithmic}[1]
					\Statex \textbf{Input:} $\{\QQQ, \MMM\}$
					\For {$\ell = n-1,\dots, 1$}
					\State Split $\MM_ \ell$ into three parts:
					\begin{itemize}
						\item $\Muno{\ell}:=\{Q\in\MM_\ell\mid Q\subset \omega_{\ell}\}$.
						\item $\Mdos{\ell}:=\{Q\in\MM_\ell\mid Q\not\subset \omega_{\ell} \wedge \parent(Q)\subset\omega_{\ell-1}\}$.
						\item $\Mtres{\ell}:=\{Q\in\MM_\ell\mid Q\not\subset \omega_{\ell} \wedge \parent(Q)\not\subset\omega_{\ell-1}\}$.
					\end{itemize}
					\State $\MM_{\ell-1} \gets \MM_{\ell-1} \cup\{\parent(Q)\mid Q\in\Mtres{\ell}\}$
					\EndFor
					\State $\Muno{0} \gets \MM_{0}$
					\State $\Mdos{0}\gets \emptyset$
					\Statex \textbf{Output:} $\MMMuno{} = \{\Muno{\ell}\}_{\ell=0}^{n-1}, \; \MMMdos{} = \{\Mdos{\ell}\}_{\ell=0}^{n-1}$
				\end{algorithmic}
			\end{algorithm}

			\begin{remark}\label{ob: cardinal de M}
				Note that the initial number of marked cells in $\MMM$ and the number of cells obtained from \Cref{alg:update_marked_elements} satisfy
				\(
				\#(\MMMuno{}\cup\MMMdos{}) \leq \# \MMM.
				\)
				This is because when two marked active cells share the same lower-level parent, \Cref{alg:update_marked_elements} replaces them with the same single cell.
			\end{remark}

			The following proposition outlines the key properties of $\MMMuno{}$ and $\MMMdos{}$. Their proof is straightforward from the construction in Algorithm~\ref{alg:update_marked_elements} and is therefore omitted.

			\begin{proposition}\label{Prop: celdas M1 y M2}
				Let $\MMM = \{\MM_\ell\}_{\ell=0}^{n-1}$ be a set of active cells in $\QQQ$ and $\MMMuno{} = \{\Muno{\ell}\}_{\ell=0}^{n-1}$, $\MMMdos{} = \{\Mdos{\ell}\}_{\ell=0}^{n-1}$ be the outputs of \Cref{alg:update_marked_elements}. Then:
				\begin{enumerate}[label=(\roman*)]
					\item \label{item1 Prop: celdas M1 y M2} The cells in $\Muno{\ell}$ are active cells of level $\ell$ and have approximation power $\ell$.
					\item \label{item2 Prop: celdas M1 y M2} The cells in $\Mdos{\ell}$ are active or deactivated cells of level $\ell$ and have approximation power $\ell-1$.
				\end{enumerate}
			\end{proposition}

			After redefining $\MMM$ via~\Cref{alg:update_marked_elements}, the subsequent section addresses the selection of a potentially larger set of cells for actual refinement. This process generates the new mesh $\QQQ^*$,increasing the approximation power of the spline space on the initially marked cells while preserving the underlying structure of $\QQQ$.

			\subsection{Design, description and analysis of the algorithm}

			In this section we develop a procedure to refine a WAHM $\QQQ$ so that all active cells selected in $\MMM$ increase their approximation power by at least one with respect to the mesh $\QQQ$, and the new mesh $\QQQ^*$ remains weakly admissible.

			After applying \Cref{alg:update_marked_elements} to $\MMM$ we will proceed as follows. For each $\ell=0,1,\dots, n-1$, we will refine all cells in $\Muno{\ell}$ (and some additional ones) so that their children achieve an approximation power of $\ell+1$ (see Figure \ref{Fig: M_opt descripcion}).

			\begin{figure}[H]
				\centering

				\begin{minipage}{0.3\textwidth}
					\centering
					\begin{tikzpicture}[scale=3]

						\fill[gray! 30] (5/8,0) rectangle (1,3/8);

						\fill[gray! 80] (11/16,0) rectangle (1,5/16);

						\fill[red!] (11/16,2/8) rectangle (12/16,5/16);

						\draw (0,0) rectangle (1,1);
						\foreach \x in {1/4, 2/4, 3/4}{
							\draw (\x,0) -- (\x,1);
							\draw (0,\x) -- (1,\x);}

						\foreach \x in {3/8, 5/8, 6/8, 7/8}{
							\draw (\x,0) -- (\x,3/4);
							\draw (1/4,1-\x) -- (1,1-\x);}

						\foreach \x in {9/16, 11/16, 13/16, 14/16, 15/16}{
							\draw (\x,0) -- (\x,4/8);
							\draw (4/8,1-\x) -- (1,1-\x);}

					\end{tikzpicture}
					\vspace{0.2em}

				\end{minipage}
				\hspace{1em}
				\begin{minipage}{0.3\textwidth}
					\centering
					\begin{tikzpicture}[scale=3]

						\fill[red!] (0,0) rectangle (1,1);
						\fill[white!] (1/4,0) rectangle (1,3/4);
						\fill[red!] (3/8,0) rectangle (1,5/8);
						\fill[white!] (1/2,0) rectangle (1,1/2);

						\fill[gray! 30] (5/8,0) rectangle (1,3/8);
						\fill[gray! 80] (11/16,0) rectangle (1,5/16);

						\fill[red!] (9/16,1/8) rectangle (7/8,7/16);

						\draw (0,0) rectangle (1,1);
						\foreach \x in {1/4, 2/4, 3/4}{
							\draw (\x,0) -- (\x,1);
							\draw (0,\x) -- (1,\x);}

						\foreach \x in {3/8, 5/8, 6/8, 7/8}{
							\draw (\x,0) -- (\x,3/4);
							\draw (1/4,1-\x) -- (1,1-\x);}

						\foreach \x in {9/16, 11/16, 13/16, 14/16, 15/16}{
							\draw (\x,0) -- (\x,4/8);
							\draw (4/8,1-\x) -- (1,1-\x);}

					\end{tikzpicture}
					\vspace{0.2em}

				\end{minipage}
				\hspace{1em}
				\begin{minipage}{0.3\textwidth}
					\centering
					\begin{tikzpicture}[scale=3]

						\fill[gray! 30] (0,0) rectangle (1,1);

						\fill[gray! 80] (9/16,0) rectangle (1,7/16);

						\fill[gray! 120] (21/32,7/32) rectangle (25/32,11/32);

						\draw (0,0) rectangle (1,1);

						\foreach \x in {1/4, 2/4, 3/4}{
							\draw (\x,0) -- (\x,1);
							\draw (0,\x) -- (1,\x);}

						\foreach \x in {1/8, 3/8, 5/8, 6/8, 7/8}{
							\draw (\x,0) -- (\x,1);
							\draw (0,1-\x) -- (1,1-\x);}

						\foreach \x in {7/16, 9/16, 11/16, 13/16, 14/16, 15/16}{
							\draw (\x,0) -- (\x,5/8);
							\draw (3/8,1-\x) -- (1,1-\x);}

						\foreach \x in {19/32, 21/32, 23/32, 25/32, 27/32}{
							\draw (\x,1/8) -- (\x,14/32);
							\draw (9/16,1-\x) -- (7/8,1-\x);}

					\end{tikzpicture}
					\vspace{0.2em}

				\end{minipage}

				\caption{Considering $\mathbf{p}=(3,3)$, the initial mesh, on the left, is weakly admissible and contains a marked cell with optimal approximation power. The central figure indicates the cells that must be refined so that the initially marked cell increases its approximation power in the new mesh while ensuring that the mesh remains weakly admissible. Finally, the figure on the right shows the resulting mesh after this refinement process.}
				\label{Fig: M_opt descripcion}
			\end{figure}

			We must also refine the cells of level $\ell-1$ surrounding the cells in $\Mdos{\ell}$ so that, upon completion, they increase their approximation power. Note that marking a cell does not guarantee that it will be refined during the process. This is the case for the cells in $\MMMdos{}$ (see Figure \ref{Fig: M_subopt descripcion}). However, they might be refined as a consequence of processing another marked cell.

			\begin{figure}[ht]
				\centering

				\begin{minipage}{0.3\textwidth}
					\centering
					\begin{tikzpicture}[scale=3]

						\fill[gray! 30] (5/8,0) rectangle (1,3/8);

						\fill[gray! 80] (11/16,0) rectangle (1,5/16);

						\fill[red!] (10/16,5/16) rectangle (11/16,6/16);

						\draw (0,0) rectangle (1,1);
						\foreach \x in {1/4, 2/4, 3/4}{
							\draw (\x,0) -- (\x,1);
							\draw (0,\x) -- (1,\x);}

						\foreach \x in {3/8, 5/8, 6/8, 7/8}{
							\draw (\x,0) -- (\x,3/4);
							\draw (1/4,1-\x) -- (1,1-\x);}

						\foreach \x in {9/16, 11/16, 13/16, 14/16, 15/16}{
							\draw (\x,0) -- (\x,4/8);
							\draw (4/8,1-\x) -- (1,1-\x);}

					\end{tikzpicture}
					\vspace{0.2em}

				\end{minipage}
				\hspace{1em}
				\begin{minipage}{0.3\textwidth}
					\centering
					\begin{tikzpicture}[scale=3]

						\fill[red!] (0,0) rectangle (1,1);
						\fill[white!] (1/4,0) rectangle (1,3/4);
						\fill[red!] (3/8,0) rectangle (1,5/8);
						\fill[white!] (1/2,0) rectangle (1,1/2);
						\fill[white!] (3/8,0) rectangle (4/8,1/8);
						\fill[white!] (7/8,1/2) rectangle (1,5/8);

						\fill[gray! 30] (5/8,0) rectangle (1,3/8);
						\fill[gray! 80] (11/16,0) rectangle (1,5/16);

						\draw (0,0) rectangle (1,1);
						\foreach \x in {1/4, 2/4, 3/4}{
							\draw (\x,0) -- (\x,1);
							\draw (0,\x) -- (1,\x);}

						\foreach \x in {3/8, 5/8, 6/8, 7/8}{
							\draw (\x,0) -- (\x,3/4);
							\draw (1/4,1-\x) -- (1,1-\x);}

						\foreach \x in {9/16, 11/16, 13/16, 14/16, 15/16}{
							\draw (\x,0) -- (\x,4/8);
							\draw (4/8,1-\x) -- (1,1-\x);}

					\end{tikzpicture}
					\vspace{0.2em}

				\end{minipage}
				\hspace{1em}
				\begin{minipage}{0.3\textwidth}
					\centering
					\begin{tikzpicture}[scale=3]

						\fill[gray! 30] (0,0) rectangle (1,1);

						\fill[gray! 80] (11/16,0) rectangle (1,5/16);

						\fill[gray! 80] (9/16,5/16) rectangle (11/16,7/16);

						\draw (0,0) rectangle (1,1);

						\foreach \x in {1/4, 2/4, 3/4}{
							\draw (\x,0) -- (\x,1);
							\draw (0,\x) -- (1,\x);}

						\foreach \x in {1/8, 3/8, 5/8, 6/8, 7/8}{
							\draw (\x,0) -- (\x,1);
							\draw (0,1-\x) -- (1,1-\x);}

						\foreach \x in {9/16, 11/16, 13/16, 14/16, 15/16}{
							\draw (\x,0) -- (\x,4/8);
							\draw (4/8,1-\x) -- (1,1-\x);}

						\foreach \x in {7/16, 9/16, 11/16, 13/16, 14/16}{
							\draw (\x,1/8) -- (\x,5/8);
							\draw (3/8,1-\x) -- (7/8,1-\x);}

					\end{tikzpicture}
					\vspace{0.2em}

				\end{minipage}
				\caption{Considering $\mathbf{p}=(3,3)$, the initial mesh, on the left, contains a cell with suboptimal approximation power. The central figure shows the cells that must be refined to increase the approximation power of the marked cell and preserve the weakly admissible property in the new mesh. The figure on the right presents the mesh obtained after this refinement.}
				\label{Fig: M_subopt descripcion}
			\end{figure}

			We now describe the proposed method to iteratively obtain the active cells of the mesh $\QQQ$ that must be refined.

			\begin{method}[Adaptive refinement procedure]\label{met: procedimiento adaptativo}
				\leavevmode
				\begin{description}
					\item[Initialize:] Let $\QQQ$ be a WAHM constructed from a clustered hierarchy of subdomains $\mathbf{\Omega}_n = \{\Omega_\ell\}_{\ell=0}^n$, and let $\MMM$ be a subset of active cells. We consider the subsets $\MMMuno{} = \{\Muno{\ell}\}_{\ell=0}^{n-1}$ and $\MMMdos{} = \{\Mdos{\ell}\}_{\ell=0}^{n-1}$ obtained by applying \Cref{alg:update_marked_elements} to $\MMM$. We set $\Omega_{n+1}^*:=\emptyset$ and $\Omega_{0}^* :=\Omega_0$.

					\item[Iterative Step:] For $\ell = n, n-1, \ldots, 1$, we define
					\begin{equation}\label{Eq: M_ADM nivel l}
						\Mcuatro{\ell} := \{\parent(Q) \mid Q \in \QQ_{\ell+1}, Q \subset \omega_{\ell+1}^* \setminus \omega_{\ell}\},
					\end{equation}
					noting that $\Mcuatro{n} = \emptyset$. To collect all cells requiring refinement at level $\ell-1$, we define:
					\begin{equation}\label{Eq: W_ell^*}
						W_\ell^* := \children(\Muno{\ell-1}) \cup \Mdos{\ell} \cup \Mcuatro{\ell}.
					\end{equation}
					Then, the updated subdomain is defined as:
					\begin{equation} \label{Def: Omega_ell*}
						\Omega_\ell^* := \Omega_\ell \cup \bigcup_{Q \in W_\ell^*} (\mathcal{N}_Q \cap \Omega_0),
					\end{equation}
					which consequently defines $\omega_\ell^*$.
				\end{description}
			\end{method}

			This unified procedure seamlessly handles all refinement requirements. The set $W_\ell^*$ gathers three types of cells that must be included in $\omega_\ell^*$:
			\begin{enumerate}
				\item $\children(\Muno{\ell-1})$: Children of marked optimal cells, ensuring they achieve an approximation power of $\ell$. (Note that for $\ell=n$, $W_n^*$ reduces entirely to this set, as $\Mdos{n}=\emptyset$).
				\item $\Mdos{\ell}$: Suboptimal marked cells whose approximation power must be increased.
				\item $\Mcuatro{\ell}$: Parent cells needed strictly to preserve the weakly admissible property. Since $\omega_{\ell+1}^* \setminus \omega_\ell \subset \omega_\ell^*$ guarantees $\omega_{\ell+1}^* \subset \omega_\ell^*$, including $\Mcuatro{\ell}$ in $W_\ell^*$ is sufficient.
			\end{enumerate}

			To ensure that every cell $Q \in W_\ell^*$ is effectively included in $\omega_\ell^*$, \eqref{Def: Omega_ell*} expands $\Omega_\ell$ by adding the required cells of level $\ell-1$ from the neighborhoods $\mathcal{N}_Q$ (see \Cref{Fig: conjuntos W y Z}).

			\begin{figure}[H]
					\centering
				\begin{minipage}{0.3\textwidth}
					\centering
					\begin{tikzpicture}[scale=2.5]
						\fill[gray!50] (0,0) rectangle (1.0,1.0);

						\fill[red] (0.4,0.4) rectangle (0.6,0.6);
						\draw[very thick] (0,0) rectangle (1,1);

						\foreach \x in {0.2, 0.4, 0.6, 0.8} {
							\draw (\x,0) -- (\x,1);
							\draw (0,\x) -- (1,\x);
						}

						\draw (0.5,0.4) -- (0.5,0.6);
						\draw (0.4,0.5) -- (0.6,0.5);
					\end{tikzpicture}
					\vspace{0.2em}

				\end{minipage}
				\hspace{1em}
				\begin{minipage}{0.4\textwidth}
					\centering
					\begin{tikzpicture}[scale=2.5]

						\fill[red! 50] (0.3,0.3) rectangle (0.7,0.7);
						\fill[red!] (0.4,0.4) rectangle (0.6,0.6);

						\draw[very thick] (0,0) rectangle (1,1);

						\foreach \x in {0.1, 0.2,..., 0.9} {
							\draw[gray! 50] (\x,0) -- (\x,1);
							\draw[gray! 50] (0,\x) -- (1,\x);
						}

						\fill[red] (0.3,0.3) rectangle (0.7,0.7);

						\foreach \x in {0.2, 0.4, 0.6, 0.8} {
							\draw (\x,0) -- (\x,1);
							\draw (0,\x) -- (1,\x);
						}

						\draw (0.5,0.4) -- (0.5,0.6);
						\draw (0.4,0.5) -- (0.6,0.5);

					\end{tikzpicture}
					\vspace{0.2em}

				\end{minipage}
				\caption{Cutaway of the mesh at levels $n-1$ and $n$ for $\mathbf{p}=(3,3)$. Left: A marked cell $Q' \in \Muno{n-1}$ at level $n-1$ with its children $W_n^*$. The set of cells of level $n-1$ contained in $\mathcal{N}_{Q}\cap\Omega_0$ for each $Q\in W_n^*$ is also depicted. Right: After refinement, the cells of $W_n^*$ are successfully included in $\omega^*_n$.}\label{Fig: conjuntos W y Z}
			\end{figure}

			\begin{remark}\label{Ob: algoritmo 2 refinamiento}
				Using \Cref{met: procedimiento adaptativo} yields the following properties for all valid $\ell$:
				\begin{enumerate}[label=(\roman*)]
					\item \label{item1 Ob: algoritmo 2 refinamiento} $\Omega_\ell \subset \Omega_\ell^*$ and $\omega_\ell \subset \omega_\ell^*$.
					\item \label{item3 Ob: algoritmo 2 refinamiento} If $Q \in \QQ_\ell$ and $Q \subset \omega_\ell^* \setminus \omega_{\ell-1}$, then $Q \in W_{\ell-1}^*$. This follows directly from \eqref{Eq: M_ADM nivel l} and \eqref{Eq: W_ell^*}.
				\end{enumerate}
			\end{remark}

			We now guarantee that the sequence $\mathbf{\Omega}_{n+1}^*$ forms a valid hierarchical refinement.

			\begin{lemma}[Well-definedness of $\mathbf{\Omega}^*_{n+1}$]\label{Lemma: buena definicion de la malla}
				Let $\QQQ$ be a WAHM associated with a clustered hierarchy $\mathbf{\Omega}_n$. If $\mathbf{\Omega}_{n+1}^*$ is obtained via \Cref{met: procedimiento adaptativo}, then $\Omega^*_\ell \subset \Omega_{\ell-1}$ for all $\ell = n+1, \dots, 1$.
			\end{lemma}

			\begin{proof}
				We proceed by backward induction. For the base case $\ell=n$, if $Q \in W_n^*$, then $\parent(Q) \in \Muno{n-1}$ and, by~\Cref{Prop: celdas M1 y M2}, $\parent(Q)\subset\omega_{n-1}$. Taking into account \Cref{parent_extendido_Q_incluido_extendido_Parent_Q}, $\mathcal{N}_Q \cap \Omega_0 \subset \widetilde{\parent(Q)} \subset \Omega_{n-1}$, implying $\Omega_n^* \subset \Omega_{n-1}$.
				Assume $\Omega^*_{\ell+1} \subset \Omega_{\ell}$ holds for some $\ell$. To prove $\Omega^*_\ell \subset \Omega_{\ell-1}$, it suffices by \eqref{Def: Omega_ell*} to show that $\mathcal{N}_Q \cap \Omega_0 \subset \Omega_{\ell-1}$ for all $Q \in W_\ell^*$.
				Since $\QQQ$ is a WAHM, \Cref{Equivalencia_Mallas_DA} implies $\mathcal{N}_Q \cap \Omega_0 \subset \Omega_{\ell-1}$ for any $Q \subset \Omega_\ell$. Both $\Mdos{\ell}$ and $\Mcuatro{\ell}$ consist of cells included in $\Omega_\ell$ (the latter due to the inductive hypothesis $\Omega^*_{\ell+1} \subset \Omega_{\ell}$). Thus, the condition holds for $Q \in \Mdos{\ell} \cup \Mcuatro{\ell}$. Finally, for $Q \in \children(\Muno{\ell-1})$, the same reasoning used in the base case yields $\mathcal{N}_Q \cap \Omega_0 \subset \Omega_{\ell-1}$.
			\end{proof}

			The following theorem collects the main properties of the refined mesh $\QQQ^*$, proving that the approximation power of the marked cells increases while the clustered and weakly admissible structure is maintained.

			\begin{theorem}[Main result]\label{Teo: princial refinamiento}
				Let $\QQQ$ be a WAHM associated with a clustered hierarchy $\mathbf{\Omega}_n$. Let $\QQQ^*$ be generated by $\mathbf{\Omega}_{n+1}^*$ via \Cref{met: procedimiento adaptativo}. Then,
				\begin{enumerate}[label=(\roman*)]
					\item \label{item1 Teo: princial refinamiento} $\mathbf{\Omega}_{n+1}^*$ is a clustered hierarchy of subdomains.
					\item \label{item2 Teo: princial refinamiento} For $\ell < n$, if $Q \in \Muno{\ell}$, then $Q \in \AAA_{\ell+1}(\QQQ^*)$.
					\item \label{item3 Teo: princial refinamiento} For $\ell < n$, if $Q \in \Mdos{\ell}$, then $Q \in \AAA_k(\QQQ^*)$ for some $k\geq \ell $.
					\item \label{item4 Teo: principal refinamiento} If $Q$ is a cell in $\MMM$ and $Q \in \AAA_k(\QQQ)$, then $Q \in \AAA_{k'}(\QQQ^*)$ for some $ k'\geq k + 1 $.
					\item \label{item5 Teo: princial refinamiento} $\QQQ^*$ is a WAHM.
				\end{enumerate}
			\end{theorem}

			\begin{proof}
				Before proceeding, we observe that if $Q \in W_\ell^*$, then $Q \subset \omega_\ell^*$. This follows from \Cref{Teo: Omega_ell y omega_ell}, as $\mathcal{N}_Q \cap \Omega_0 \subset \Omega_\ell^*$ by construction.

				\ref{item1 Teo: princial refinamiento} Applying \Cref{Teo: jerarquia de subdominio agrupada} to $\mathbf{\Omega}_n$ allows us to rewrite \eqref{Def: Omega_ell*} as $\Omega_\ell^* = \bigcup_{Q \in A} (\mathcal{N}_Q \cap \Omega_0)$ with $A = \{Q \in \QQ_\ell \mid Q \subset \omega_\ell \lor Q \in W_\ell^*\}$. Since $\mathcal{N}_Q \in \mathbf{F}^\infty_{\mathbf{p},\ell-1}$ (\Cref{Prop: C_Q es Forma_p}), $\mathbf{\Omega}_{n+1}^*$ is clustered.

			\ref{item2 Teo: princial refinamiento} If $Q \in \Muno{\ell}$, all its children belong to $W_{\ell+1}^*$. By our initial observation, they are included in $\omega_{\ell+1}^*$, meaning $Q \subset \omega_{\ell+1}^*$.  Also, $Q$ is an active cell of $\QQQ$ by \Cref{Prop: celdas M1 y M2}, and hence $Q \not\subset \Omega_{\ell+1}$. For $j \geq \ell+2$, we have $\omega_j^* \subset \Omega_j^* \subset \Omega_{j-1} \subset \Omega_{\ell+1}$ by \Cref{Lemma: buena definicion de la malla} and the nesting of $\mathbf{\Omega}_n$, so that $Q \not\subset \omega_j^*$ for any $j\geq \ell+2$. Thus, $Q \in \AAA_{\ell+1}(\QQQ^*)$.

				\ref{item3 Teo: princial refinamiento} If $Q \in \Mdos{\ell}$, then $Q \in W_\ell^*$. Consequently, $Q \subset \omega_\ell^*$, implying $Q\in\AAA_k(\QQQ^*)$ for some $k\geq\ell$.

				\ref{item4 Teo: principal refinamiento} Let $Q \in \MM_\ell \cap \AAA_k(\QQQ)$ with $k \leq \ell$.
				If $k=\ell$, then $Q \in \Muno{\ell}$ and the result follows from \ref{item2 Teo: princial refinamiento}.
				If $k=\ell-1$, then $Q \in \Mdos{\ell}$ and the result follows from \ref{item3 Teo: princial refinamiento}.
				If $k < \ell-1$, then $\parent_{k}(Q)\subset \omega_{k}$ and $\parent_{k+1}(Q)\not\subset \omega_{k+1}$ and thus $\parent_{k+1}(Q) \in \Mdos{k+1}$. By \ref{item3 Teo: princial refinamiento}, $\parent_{k+1}(Q) \in \AAA_{k'}(\QQQ^*)$ for some $k'\geq k+1$. Since $Q \subset \parent_{k+1}(Q)$, its approximation power increases at least to $k+1$.

				\ref{item5 Teo: princial refinamiento} To verify $\QQQ^*$ is a WAHM, consider $Q \subset \omega_\ell^* \setminus \omega_{\ell-1}$. By \Cref{Ob: algoritmo 2 refinamiento}\ref{item3 Ob: algoritmo 2 refinamiento}, $Q \in W_{\ell-1}^*$. This implies $Q \subset \omega_{\ell-1}^*$. Hence, $\omega_\ell^* \setminus \omega_{\ell-1} \subset \omega_{\ell-1}^*$, which combined with $\omega_{\ell-1} \subset \omega_{\ell-1}^*$ yields $\omega_\ell^* \subset \omega_{\ell-1}^*$.
			\end{proof}

			\begin{remark}\label{Ob: rango poder aproximacion}
				Item~\ref{item3 Teo: princial refinamiento} of \Cref{Teo: princial refinamiento} can be made more precise. Let $Q\in\Mdos{\ell}$ and let $j\geq\ell$ be such that $Q\subset\Omega_j$ and $Q\not\subset\Omega_{j+1}$. Arguing as in the proof of item~\ref{item2 Teo: princial refinamiento}, we have $Q\not\subset\omega^*_i$ for all $i\geq j+2$, and hence $Q\in\AAA_k(\QQQ^*)$ for some $k$ with $\ell\leq k\leq j+1$. In particular, if $j=\ell$, then $Q$ is active, which implies that $Q\in\AAA_\ell(\QQQ^*)$ or $Q\in\AAA_{\ell+1}(\QQQ^*)$.

			\end{remark}

			\subsection{Algorithm implementation}

			This section details the implementation of the adaptive refinement procedure described in \Cref{met: procedimiento adaptativo}. The core of this process is the marking strategy, which selects the cells requiring refinement. Once identified, these cells are used to generate a new hierarchy of subdomains, and consequently, the updated mesh and hierarchical space according to \Cref{def: construccion de refinamiento}.

			Given a weakly admissible hierarchical mesh $\QQQ$ and a subset of active cells $\MMM$, the updated set of marked cells $\hat{\MMM} = \{\hat{\MM}_\ell\}_{\ell=0}^{n-1}$ is computed via the primary marking function:
			\[
			\hat{\MMM} = \textbf{approximation\_power\_marking}(\QQQ, \MMM).
			\]
			Before detailing this algorithm, we introduce two essential routines, implemented using standard functions from the GeoPDEs package \cite{GEOPDES-NEW}:

			\begin{enumerate}[label=(\roman*)]\label{R: Rutinas}
				\item \label{R: rutina CQ de una celda} \textbf{Computation of the neighborhood $\mathcal{N}_Q$:} For a cell $Q \in \QQ_\ell$, the routine
				\[
				\MM^Q_{\ell-1} = \textbf{mark\_for\_increasing\_order\_on}(Q)
				\]
				identifies the cells of level $\ell-1$ comprising $\mathcal{N}_Q \cap \Omega_0$. These are precisely the cells that must be refined to ensure $Q \subset \omega^*_\ell$. In GeoPDEs, this is achieved by composing internal functions:
				\[
				\MM^Q_{\ell-1} = \textbf{get\_parent\_of\_cell}(\textbf{get\_cells}(\textbf{get\_basis\_functions}(Q))).
				\]

				This composition first retrieves the B-splines at level $\ell$ that are non-zero on $Q$, then finds all cells of level $\ell$ in their supports, and finally extracts their parents of level $\ell-1$.

				\item \label{R: rutina omega de un conjunto} \textbf{Computation of cells with full approximation power:} Given an arbitrary set $B$ of cells of level $\ell$, the routine
				\[
				\omega_\ell(B) = \textbf{compute\_cells\_with\_full\_approximation}(B)
				\]
				extracts the subset of cells in $B$ whose support extension is fully contained within $B$. For instance, $\omega_\ell(\Omega_\ell)$ recovers all cells of level $\ell$ in $\omega_\ell$, and $\omega_\ell(\mathcal{N}_Q)$ yields the core set $\core(\mathcal{N}_Q)$ defined in \Cref{Prop: representacion de supp B-splines}. We also remark that the support extension of a cell $Q$ is computed via $\EE_\ell^Q = \textbf{get\_cells}(\textbf{get\_basis\_functions}(Q))$.
			\end{enumerate}

			With these tools, we summarize the implementation of \Cref{met: procedimiento adaptativo} in \Cref{alg:compute_marked_elements}.

			\begin{algorithm}[htb]
				\small
				\caption[\textbf{approximation\_power\_marking}]{\protect\textbf{approximation\_power\_marking}}\label{alg:compute_marked_elements}
				\begin{algorithmic}[1]
					\Statex \textbf{Input:} Mesh $\QQQ$, initial marked cells $\MMM$
					\State $\{\MMMuno{}, \MMMdos{}\} \gets \textbf{update\_marked\_elements}(\QQQ, \MMM)$ \Comment{Via \Cref{alg:update_marked_elements}}
					\State $\Mcuatro{n} \gets \emptyset$
					\For{$\ell = n, \dots, 1$}
					\State $\{\hat{\MM}_{\ell-1}, \Mcuatro{\ell-1}\} \gets \textbf{compute\_cells\_to\_refine\_in\_level}(\Muno{\ell-1}, \Mdos{\ell}, \Mcuatro{\ell})$
					\EndFor
					\Statex \textbf{Output:} Updated marked set $\hat{\MMM} = \{\hat{\MM}_\ell\}_{\ell=0}^{n-1}$
				\end{algorithmic}
			\end{algorithm}

			\Cref{alg:compute_marked_elements} begins by sorting the initial set $\MMM$ into the structurally significant subsets $\MMMuno{}$ and $\MMMdos{}$. The main loop (lines 3--5) iterates downwards through the levels. At each step $\ell$, the routine \textbf{compute\_cells\_to\_refine\_in\_level} performs two tasks:

			First, it computes $\hat{\MM}_{\ell-1}$, the set of active cells level of $\ell-1$ that must be refined to construct the updated subdomain $\Omega_\ell^*$. Following \Cref{met: procedimiento adaptativo}, it defines $W_\ell^*$ as in~\eqref{Eq: W_ell^*} and applies routine \ref{R: rutina CQ de una celda} to construct:
			\[
			\hat{\MM}_{\ell-1} := \left( \bigcup_{Q \in W_\ell^*} \MM^Q_{\ell-1} \right) \cap \QQQ.
			\]

			Second, it computes the set $\Mcuatro{\ell-1}$ required for the next iteration to preserve weak admissibility:
			\[
			\Mcuatro{\ell-1} = \{\parent(Q) \mid Q \in \QQ_\ell, \, Q \subset \omega_\ell^* \setminus \omega_{\ell-1}\}.
			\]
			Here, $\omega_\ell^*$ is evaluated using routine \ref{R: rutina omega de un conjunto} as $\textbf{compute\_cells\_with\_full\_approximation}(\Omega_\ell^*)$. While conceptually straightforward, computing $\Mcuatro{\ell-1}$ this way requires constructing the full set $\Omega_\ell^*$ at each step. A more efficient strategy to obtain $\Mcuatro{\ell-1}$ will be detailed later in the context of \Cref{alg: MARK_RECURSIVE}.

			Finally, upon completion of \Cref{alg:compute_marked_elements}, the output $\hat{\MMM}$ is passed to the mesh refinement module. Applying \Cref{def: construccion de refinamiento}, the updated hierarchical mesh $\QQQ^*$ and its associated basis $\HH^*$ are generated via:
			\[
			\QQQ^* = \textbf{refine\_hierarchical\_mesh}(\QQQ, \hat{\MMM}).
			\]

			\subsection{Refinement complexity}\label{Sec: Complejidad}

			We conclude this section by studying the complexity of an iterative procedure based on the proposed algorithm. Since the algorithm refines, besides the marked cells, the additional cells needed to increase their approximation power and to preserve weak admissibility, it is important to control the size of the resulting meshes. Specifically, we bound the growth of the number of mesh cells relative to the number of marked elements, that is,
			\begin{equation*} \label{eq: bound_goal}
				\#\QQQ_J - \#\QQQ_0 \leq C \sum_{j=0}^{J-1} \#\MMM_j,
			\end{equation*}
			where $\#\MMM_j$ is the number of marked cells in the current mesh $\QQQ_j$ at iteration $j$, and the constant $C$ depends only on the dimension $d$ and the polynomial degree $\mathbf{p} = (p, \dots, p)$.

			The main results of this section build upon the ideas developed for standard adaptive finite element methods, where bounds of this type are a key ingredient in the analysis of optimal convergence rates (see e.g.~\cite[Theorem 4.3]{NSV2009} and \cite{NV2012}) and the lines presented in \cite[Lemma 12, Theorem 13]{BGMP2016} for hierarchical splines on strictly admissible meshes of class $m$.

			To facilitate the analysis, we introduce \Cref{alg: MARK_WEAKLY}, a recursive implementation of the marking strategy previously detailed in \Cref{alg:compute_marked_elements}.

			\begin{algorithm}[H]
				\small
				\caption[\textbf{approximation\_power\_marking\_recursive}]{\protect\textbf{approximation\_power\_marking\_recursive}}\label{alg: MARK_WEAKLY}
				\begin{algorithmic}[1]
					\Statex \textbf{Input:} Mesh $\QQQ$, initial marked cells $\MMM$
					\State $\{\MMMuno{}, \MMMdos{}\} \gets \textbf{update\_marked\_elements}(\QQQ, \MMM)$ \Comment{Via \Cref{alg:update_marked_elements}}
					\State $\hat{\MMM} = \{\hat{\MM}_\ell\}_{\ell=0}^{n-1} \gets \emptyset$
					\For{$\ell = n, \dots, 1$}
					\For{$Q \in \children(\Muno{\ell-1}) \cup \Mdos{\ell}$}
					\State $\hat{\MMM} \gets \textbf{mark\_recursive}(\QQQ, Q, \hat{\MMM})$
					\EndFor
					\EndFor
					\Statex \textbf{Output:} Updated marked set $\hat{\MMM}$
				\end{algorithmic}
			\end{algorithm}

			\Cref{alg: MARK_WEAKLY} applies the core routine, \textbf{mark\_recursive} (\Cref{alg: MARK_RECURSIVE}), to each cell in $\children(\Muno{\ell-1}) \cup \Mdos{\ell}$ across all levels. This recursive procedure systematically identifies all cells that must be added to $\hat{\MMM}$ to increase the approximation power on the marked cells and to maintain weak admissibility.

			\begin{algorithm}[H]
				\small
				\caption{\textbf{mark\_recursive}}\label{alg: MARK_RECURSIVE}
				\begin{algorithmic}[1]
					\Statex \textbf{Input:} Mesh $\QQQ$, cell $Q \in \QQ_\ell$, current set $\hat{\MMM}$
					\For{$Q_i \in \omega_\ell(\mathcal{N}_Q \cap \Omega_0)$}
					\If{$Q_i \not\subset \omega_{\ell-1}$}
					\State $\hat{\MMM} \gets \textbf{mark\_recursive}(\QQQ, \parent(Q_i), \hat{\MMM})$ \Comment{Replaces $\Mcuatro{\ell-1}$ computation}
					\EndIf
					\EndFor
					\State Add all $Q' \in \QQQ$ such that $Q' \subset \mathcal{N}_Q$ to $\hat{\MM}_{\ell-1}$. \Comment{Replaces $\hat{\MM}_{\ell-1}$ computation}
					\Statex \textbf{Output:} Updated set $\hat{\MMM}$
				\end{algorithmic}
			\end{algorithm}

			\Cref{alg: MARK_RECURSIVE} utilizes the set $\omega_\ell(\mathcal{N}_Q \cap \Omega_0)$ defined via Routine \ref{R: rutina omega de un conjunto}. For each $Q_i$ in this set, if $Q_i \not\subset \omega_{\ell-1}$, it implies $Q_i \in \omega_\ell^* \setminus \omega_{\ell-1}$. To ensure this cell is properly supported by $\omega_{\ell-1}^*$, the recursion acts on $\parent(Q_i)$, elegantly replacing the explicit computation of $\Mcuatro{\ell-1}$ from \Cref{alg:compute_marked_elements}. Finally, the cells of level $\ell-1$ contained in $\mathcal{N}_Q$ are stored in $\hat{\MM}_{\ell-1}$.

			Following~\cite{BGV2018}, we require auxiliary results regarding the new cells generated by \Cref{alg: MARK_RECURSIVE}. Let $\dist(Q,Q')$ denote the Euclidean distance between the midpoints of cells $Q$ and $Q'$, and $\lev(Q)$ denote the hierarchical level of $Q$.

			\begin{remark}\label{ob: niveles refine_recursive}
				Let $Q \in \QQ_\ell$ and consider the set $\hat{\MMM}_Q = \textbf{mark\_recursive}(\QQQ, Q, \emptyset)$. If $\QQQ^* = \textbf{refine\_hierarchical\_mesh}(\QQQ, \hat{\MMM}_Q)$, then any new active cell $Q' \in \QQQ^* \setminus \QQQ$ satisfies $\lev(Q') \leq \lev(Q)$. Indeed, the recursive process only visits ancestors of level $k \leq \ell$. At each step of the recursion, the cells marked for refinement are of level $k-1$. Consequently, the children of these marked cells (the new cells $Q'$) will have level $k \leq \ell$.
			\end{remark}

			\begin{remark} \label{Ob: distancia entre Q y Q' de CQ}
				For $Q \in \QQ^\infty_\ell$ and $Q' \in \QQ^\infty_{\ell-1}$, if $Q' \subset \mathcal{N}_Q$, then basic geometric arguments yield:
				\begin{equation*} \label{eq: 1 distancia entre celdas y constantes}
					\dist(Q,Q') \leq \sqrt{d} \, (2p+1) 2^{-\lev(Q)-1}.
				\end{equation*}
			\end{remark}

			The following key lemma adapts the approach from \cite[Lemma 12]{BGV2018} to our refinement algorithm. It bounds the spatial distance between a newly created cell and the originally marked cell that triggered its creation.

			\begin{lemma} \label{lema: Complejidad}
				Let $\QQQ$ be a WAHM, $\MMM$ the set of marked cells, and $\{\MMMuno{}, \MMMdos{}\}$ the sets obtained via \Cref{alg:update_marked_elements}. Consider a newly generated cell $Q' \in \QQQ^* \setminus \QQQ$ which satisfies $Q' \in \children(Q_0)$ for some $Q_0$.
				\begin{enumerate}[label=(\roman*)]
					\item \label{item 1:lema-complejidad} Assume that $Q_0 \in \hat{\MMM}_{Q_{ch}} = \textbf{mark\_recursive}(\QQQ, Q_{ch}, \emptyset)$ where $Q_{ch} \in \children(Q)$ and $Q \in \MMMuno{}$.
					\item \label{item 2:lema-complejidad} Assume that $Q_0 \in \hat{\MMM}_Q = \textbf{mark\_recursive}(\QQQ, Q, \emptyset)$ where $Q \in \MMMdos{}$.
				\end{enumerate}
				In both cases, it holds that
				\begin{equation*} \label{eq: lema complejidad}
					\dist(Q', Q) \leq 2^{-\lev(Q')} \tilde{C}, \quad \text{where} \quad \tilde{C} := \frac{3}{2} \sqrt{d}(2p+1).
				\end{equation*}
			\end{lemma}

			\begin{proof}
				We present a unified proof for both cases. Let $Q''$ be the cell that initiates the recursion: $Q'' = Q_{ch}$ in case \ref{item 1:lema-complejidad}, and $Q'' = Q$ in case \ref{item 2:lema-complejidad}. Since $Q'$ is generated from $Q_0$, the recursive nature of the algorithm implies the existence of a sequence of cells $\{Q_0, Q_1, \dots, Q_k = Q''\}$ such that $Q_{j-1} \subset \mathcal{N}_{Q_j} \cap \Omega_0$ and $\lev(Q_{j-1}) = \lev(Q_j) - 1$ for $j = 1, \dots, k$.

				By \Cref{Ob: distancia entre Q y Q' de CQ}, the distance between consecutive cells in this sequence is bounded by:
				\begin{equation*}\label{eq: distancia Qj Qj-1}
					\dist(Q_j, Q_{j-1}) \leq \sqrt{d} \, (2p+1) 2^{-\lev(Q_j)-1}.
				\end{equation*}
				Using the triangle inequality along the sequence, and knowing $Q' \in \children(Q_0)$, we estimate $\dist(Q', Q'')$:
				\[
				\dist(Q', Q'') \leq \dist(Q', Q_0) + \sum_{j=1}^k \dist(Q_{j-1}, Q_j) \leq \sqrt{d} \, 2^{-\lev(Q')-1} + \sum_{j=1}^k \sqrt{d} \, (2p+1) 2^{-\lev(Q_j)-1}.
				\]
				Since $\lev(Q_j) = \lev(Q') + j - 1$, we can rewrite the sum:
				\begin{equation}\label{eq: bound_Q_tilde}
					\dist(Q', Q'') \leq \sqrt{d} \, (2p+1) 2^{-\lev(Q')-1} \left( 1 + \sum_{j=1}^k 2^{-j+1} \right).
				\end{equation}

				Now, we split into the two cases to evaluate $\dist(Q', Q)$:

				\textbf{Case \ref{item 2:lema-complejidad}:} Here, $Q'' = Q$. The sum in \eqref{eq: bound_Q_tilde} is strictly bounded by $2$. Thus,
				\[
				\dist(Q', Q) = \dist(Q', Q'') \leq 3 \sqrt{d} \, (2p+1) 2^{-\lev(Q')-1} = 2^{-\lev(Q')} \tilde{C}.
				\]

				\textbf{Case \ref{item 1:lema-complejidad}:} Here, $Q'' = Q_{ch} \in \children(Q)$, so $\dist(Q'', Q) \leq \sqrt{d} \, 2^{-\lev(Q_{ch})-1}$. Noting that $\lev(Q_{ch}) = \lev(Q') + k$, we apply the triangle inequality again:
				\[
				\dist(Q', Q) \leq \dist(Q', Q'') + \dist(Q'', Q) \leq \sqrt{d} \, (2p+1) 2^{-\lev(Q')-1} \left( 1 + \sum_{j=1}^k 2^{-j+1} + 2^{-k+1} \right).
				\]
				Since $1 + \sum_{j=1}^k 2^{-j+1} + 2^{-k+1} = 3$, the bound $\dist(Q', Q) \leq 2^{-\lev(Q')} \tilde{C}$ is directly obtained.
			\end{proof}

			We conclude with the main complexity result, whose proof follows from \cite[Theorem 13]{BGMP2016} by applying \Cref{lema: Complejidad}.

			\begin{theorem}[Refinement complexity] \label{teo: complejidad de refinamiento}
				Let $\QQQ_0 = \QQ_0$ and $J\in\NN$. For $j = 1, \dots, J$, let $\MMM_{j-1} \subset \QQQ_{j-1}$ be an arbitrary subset of marked cells, and let $\QQQ_j$ be the mesh obtained refining the cells given by the output of \textbf{approximation\_power\_marking\_recursive} (cf.~\Cref{alg: MARK_WEAKLY}). Then, the overall complexity is bounded by:
				\begin{equation*}
					\#\QQQ_J - \#\QQQ_0 \leq C \sum_{j=0}^{J-1} \#\MMM_j,
				\end{equation*}
				where $C \leq 4(4\tilde{C} + 1)^d$ is a constant depending only on $d$ and $\mathbf{p}$.
			\end{theorem}

			\section{{Local error estimates and approximation power}} \label{Sec: local approcimation weakly admissible meshes}

				In \Cref{S:algorithm}, the cells to be refined are selected so that the approximation power of every marked cell increases (\Cref{Teo: princial refinamiento}). Recall that a cell $Q$ has approximation power $k$ if $k$ is the finest level such that the hierarchical space attains on $Q$ the full resolution of level $k$ (\Cref{Def: poder de aproximacion}). This notion depends only on the hierarchical mesh and the degree, and not on the function to be approximated. In this section, we relate it to the local approximation error. More precisely, we show that on a WAHM the error of a multilevel quasi-interpolant on an active cell $Q$ of level $\ell$ with approximation power $k$ is bounded in terms of $h_k$ (\Cref{cor: estimates_active_cells_by_power}). Since $h_k>h_\ell$ when $k<\ell$, the error bound on $Q$ is governed by its approximation power rather than by its size.

			We now present the construction of a quasi-interpolant operator in hierarchical spline spaces. Originally introduced by~\cite{K98} for bivariate splines on infinite uniform knot vectors, this operator was later extended to $d$-dimensional domains with open knot vectors and multiple internal knots in~\cite{BG15}. In that work, the authors established local $L^q$-approximation estimates for $1 \leq q \leq \infty$. Here we extend their results, specifically~\cite[Theorem 4.9]{BG15}, by obtaining optimal estimates in higher-order norms.

			\subsection{Quasi-interpolation in hierarchical spline spaces}

			In order to define a multilevel operator, we combine local quasi-interpolant operators $P_{\ell}$, defined at each level $\ell = 0, 1, \dots, n-1$, which satisfy certain properties. To define $P_{\ell}$, we follow the ideas and techniques in~\cite{LLM01} using a local $L^2$-projection. For each $\ell$, we consider the subset of the basis $\mathcal{B}_{\ell}$ defined by
			\begin{equation*}
				\mathcal{B}_{\ell,\omega_{\ell}}:=\left\lbrace \beta \in \mathcal{B}_{\ell} \mid \exists Q \in \mathcal{Q}_{\ell}: Q \subset \supp(\beta) \cap \omega_{\ell} \right\rbrace.
			\end{equation*}
			Let $\mathcal{S}_{\omega_{\ell}}:= \Span(\mathcal{B}_{\ell,\omega_{\ell}})\subset \mathcal{S}_{\ell}$. We then define $P_{\ell}: L^{q}(\Omega) \longrightarrow  \mathcal{S}_{\omega_{\ell}}$ for $\ell=0,\dots,n-1$, such that
			\begin{equation} \label{def: quasi-interp level ell}
				P_{\ell}(f):=\sum_{\beta \in\mathcal{B}_{\ell,\omega_{\ell}}} \lambda_{\beta}(f)\beta, \quad \forall f \in L^{q}(\Omega),
			\end{equation}
			where $\{\lambda_{\beta}\}$ are linear functionals that have local support, form a dual basis of $\mathcal{B}_{\ell ,\omega_{\ell}}$, and are $L^q$-stable as established on~\cite[Proposition 4.3]{BG15}. The following result provides a local version of~\cite[Theorem 4.4]{BG15} summarizing the properties of $P_{\ell}$, and its proof follows the same arguments. From now on, the usual $L^q$-norm in $\Omega$ is denoted by $\|\cdot\|_{q,\Omega}$. Similarly, $W^{k}_{q}(\Omega)$ denotes the standard Sobolev space equipped with the seminorm $|\cdot|_{k,q,\Omega}$.

			\begin{theorem}\label{teo: properties of P_ell}
				Let $\{ P_{\ell} \}_{\ell=0}^{n-1}$ be a set of operators defined as in \eqref{def: quasi-interp level ell}. For $\ell=0,1,\dots,n-1$, we have:
				\begin{enumerate}
					\item[(i)] $P_{\ell}$ preserves splines in $\mathcal{S}_{\omega_{\ell}}$, i.e., $P_{\ell}s = s$, for all $s\in\mathcal{S}_{\omega_{\ell}}$.
					\item[(ii)] $P_{\ell}$ is supported in $\omega_{\ell}$, that is, if $f|_{\omega_{\ell}}\equiv 0$, then $P_{\ell}f=0$.
					\item[(iii)] For all $s \in \mathcal{S}_{\ell}$, $P_{\ell}s\equiv s$ in $\omega_{\ell}$.
				\end{enumerate}
				Furthermore, let $C$ be a union of cells of level $\ell$, and denote its neighborhood by $V_C^\ell := \bigcup \{\tilde{Q}\,\mid\, Q\in\QQ_\ell,\, Q\subset C\}$. Then:
				\begin{enumerate}
					\item[(iv)] (Stability) For all $f \in L^{q}(\omega_{\ell})$ and  $C \subset \Omega_\ell$, $P_{\ell}$ satisfies $\|P_{\ell}f\|_{q,C}\leq C_{S}\|f\|_{q,\,V_C^\ell \cap \omega_{\ell}}$, where the constant $C_{S} > 0$ depends only on $p$.
					\item[(v)] (Approximation) For all $f \in W^{p+1}_{q}(\Omega)$ and $C\subset \omega_\ell$, it holds that  $\|f-P_{\ell}f\|_{q,C} \leq C_{A}h_{\ell}^{p+1}|f|_{p+1,q,V_C^\ell}$, where the constant $C_{A} > 0$ depends only on $p$ and $d$.
				\end{enumerate}
			\end{theorem}

			\begin{definition}[A hierarchical quasi-interpolant operator]
				Given the set of quasi-interpolation operators defined in \eqref{def: quasi-interp level ell} for each level of the hierarchy, we define $\Pi:\;L^{q}(\Omega)\longrightarrow \Span\mathcal{H}$ by
				\begin{equation} \label{def: operador Pi multinivel}
					\begin{cases}
						\Pi_{0}:=P_{0},\\
						\Pi_{\ell+1}:=\Pi_{\ell}+P_{\ell +1}(\Id-\Pi_{\ell}), \quad \ell=0,\dots,n-2,\\
						\Pi:=\Pi_{n-1}.
					\end{cases}
				\end{equation}
			\end{definition}

			\begin{remark}\label{obs: preserva splines en Scero}
				Note that Theorem \ref{teo: properties of P_ell} (i) implies that $P_{0}s=s$ for all $s \in \mathcal{S}_{0}$. Consequently, the definition of the operator $\Pi$ yields $\Pi s=s$ for all $s \in \mathcal{S}_{0}$. Therefore, the operator preserves splines at the initial level, and in particular, in the tensor product polynomial space $\mathbb{P}_{\mathbf{p}}$, where $\mathbf{p}=(p,\dots,p)$.
			\end{remark}

			The next result can be found in~\cite[Lemma 4.8]{BG15} and follows the ideas from~\cite[Lemma 2.2.1.d]{K98}.
			\begin{lemma}
				For $f \in L^{q}(\Omega)$ it holds that $\Pi_{\ell} f=P_{\ell}f$ in $\omega_{\ell}$ for $\ell=0,\dots,n-1$.
			\end{lemma}

			The operator $\Pi$ can be expressed in terms of the operators $P_k$ when restricted to the subdomains $\omega_{\ell}$ (cf.~\cite[Theorem 4.7]{BG15}). 
			\begin{lemma}\label{teo: Pi on omeguita ell}
				If $\omega_{n-1} \subset \omega_{n-2} \subset \ldots \subset \omega_{2} \subset \omega_{1} \subset \omega_{0}$, then
				\begin{equation*}
					\Pi f = P_{\ell}f + \sum_{k=\ell+1}^{n-1}P_{k}(f-P_{k-1}f) \quad \text{in } \omega_{\ell} \quad (\ell=0,\dots,n-1),\quad \text{	for all }\, f \in L^{q}(\Omega).
				\end{equation*}
			\end{lemma}

			\subsection{Approximation properties of a multilevel quasi-interpolant}

			Recall that we consider uniform meshes at each level constituted by hypercubes of edge length $h_{{\ell}}$, and the nested spaces are obtained by dyadic refinement. That is, $\mathcal{S}_{\ell}$ is obtained from $\mathcal{S}_{\ell-1}$ by adding a breakpoint exactly at the midpoint of each edge of the hypercubes in the level $\ell-1$ mesh, yielding $h_{\ell}= h_{\ell-1}/2$ for all $\ell=0,\dots,n-1$. We denote by $\mathcal{P}_{s}$ the space of $d$-variate polynomials of total degree at most $s$. The following inverse inequality will be key for proving the main result of this section.

			\begin{lemma}(\cite[Lemma 1]{S17})\label{lem: inverse ineq}
				Let $\Gamma$ be a convex body of $\mathbb{R}^{d}$ and let $r_{\Gamma}$ be the radius of the largest ball contained in $\Gamma$. For any $g \in \mathcal{P}_{s}$ and $1 \leq q \leq \infty$ we have
				\begin{equation*}
					\|D^{\bm{\alpha}}g\|_{q,\Gamma}\leq \frac{C}{r_{\Gamma}^{|\bm{\alpha}|}}\|g\|_{q,\Gamma}, \quad \text{for } 0\leq|\bm{\alpha}|\leq s,
				\end{equation*}
				where $C>0$ is a constant independent of $g$ and $\Gamma$.
			\end{lemma} 

			\begin{theorem}[Local approximation estimates]\label{theo: sobolev norm estimation}
				Let $\mathcal{H}$ be the set of hierarchical B-splines associated with a hierarchy of subdomains $\mathbf{\Omega}_{n}$ of depth $n$, and let $\Pi: L^{q}(\Omega) \longrightarrow \Span \mathcal{H}$ be the multiscale quasi-interpolant defined in \eqref{def: operador Pi multinivel}. If the hierarchical mesh is weakly admissible, i.e.,
				\begin{equation}\label{wahm}
					\omega _{n-1} \subset \omega _{n-2} \subset \ldots \subset \omega _{2} \subset \omega _{1} \subset \omega _{0},
				\end{equation}
				 then, for any cell $Q\in\QQ_{\ell}$ such that $Q\subset \omega_{\ell}$ ($0 \leq \ell \leq n-1$) and all $f \in W_{q}^{p+1}(\Omega)$ ($1\leq q\leq \infty$), we have
				\begin{equation}\label{eq: estimate Sobolev}
					\|D^{\bm{\alpha}}(f-\Pi f)\|_{q,Q}\leq C\;h_{\ell}^{p+1-|\bm{\alpha}|}|f|_{p+1,q, N(Q)},\qquad \forall\,\bm{\alpha},\, 0\leq |\bm{\alpha}|\leq p,
				\end{equation}
					where the neighborhood $N(Q)$ is defined as $N(Q) := \tilde{Q}$ if $Q \subset \Omega_\ell \setminus \Omega_{\ell+1}$. Otherwise, if $Q \subset \Omega_{\ell+1}$, we set $N(Q) := \widetilde{\parent(Q)}$ when $\ell \geq 1$, and $N(Q) := V^0_{\tilde{Q}}$ when $\ell = 0$. Here $V^0_{\tilde{Q}}$ is the neighborhood introduced in \Cref{teo: properties of P_ell}.
				The constant $C>0$ depends on $p$ and $d$, but is independent of $f$.
			\end{theorem}

			\begin{proof}
				Let $f \in W_{q}^{p+1}(\Omega)$ and $Q \subset \omega_{\ell}$. Let $Tf$ be the averaged Taylor polynomial of degree $p$ of $f$ over $\tilde{Q}$ (cf.~\cite[Definition 1]{S17} or~\cite[Definition 4.3.1]{BS02}). By the triangle inequality, and since $\Pi g=g$ for any $g \in \mathbb{P}_{\mathbf{p}}$, we have
				\begin{equation}\label{eq: desigualdad triang con Taylor}
					\|D^{\bm{\alpha}}(f-\Pi f)\|_{q,Q}\leq \|D^{\bm{\alpha}}(f-T f)\|_{q,Q} + \|D^{\bm{\alpha}}\Pi(Tf- f)\|_{q,Q}=:[A]+[B].
				\end{equation}
				We remark that the constant $C$ appearing in the following estimates may differ at each occurrence. Term [A] is bounded above by using standard error estimates for the Taylor polynomial~\cite[Lemma 4.3.8]{BS02}: 
				\begin{equation}\label{parteA}
					[A]\leq \|D^{\bm{\alpha}}(f-T f)\|_{q,\tilde{Q}}\leq C\;\diam(\tilde{Q})^{p+1-|\bm{\alpha}|}|f|_{p+1,q,\tilde{Q}} \leq C\; h_{\ell}^{p+1-|\bm{\alpha}|}|f|_{p+1,q,\tilde{Q}}.
				\end{equation}
				In the last inequality, we used the fact that $\tilde{Q}$ consists of at most $(p+1)^d$ cells of level $\ell$, implying $\diam(\tilde{Q}) \le C h_\ell$.
				For [B], since $(\Pi (Tf-f))|_{Q}\in \mathbb{P}_{\mathbf{p}}\subset \mathcal{P}_{s}$ with $s=d \,p$, we apply Lemma~\ref{lem: inverse ineq} to obtain:
				\begin{equation}\label{eq: bound B inverse ineq}
					[B] \leq \frac{C}{h_{\ell}^{|\bm{\alpha}|}}\|\Pi(Tf- f)\|_{q,Q},\qquad \text{for }0\leq|\bm{\alpha}|\leq p.
				\end{equation}

				We now analyze the two possible cases for $Q$. If $Q \subset \Omega_\ell \setminus \Omega_{\ell+1}$, then $\Pi f|_Q = P_\ell f|_Q$. Using the stability of $P_\ell$ (Theorem~\ref{teo: properties of P_ell} (iv)) and the Taylor polynomial estimates, \eqref{eq: bound B inverse ineq} becomes:
				\begin{equation*}
					[B] \leq \frac{C}{h_{\ell}^{|\bm{\alpha}|}}\|Tf-f\|_{q,\tilde{Q}\cap \omega_\ell} \leq C\, h_{\ell}^{p+1-|\bm{\alpha}|}\,|f|_{p+1,q,\tilde{Q}},
				\end{equation*}
				which, combined with \eqref{parteA}, proves the result for $N(Q) = \tilde{Q}$.

				On the other hand, if $Q\subset\Omega_{\ell+1}$, we use Lemma \ref{teo: Pi on omeguita ell} and the stability of $P_{\ell}$ to obtain:
				\begin{align*}
					[B] &\leq \frac{C}{h_{\ell}^{|\bm{\alpha}|}} \Big\| P_{\ell}(Tf- f)+ \sum_{k=\ell+1}^{n-1} P_{k}\big( Tf-f-P_{k-1}(Tf-f)\big) \Big\|_{q,Q} \\
					&\leq \frac{C}{h_{\ell}^{|\bm{\alpha}|}} \Big( \|Tf- f\|_{q,\tilde{Q}} + \sum_{k=\ell+1}^{n-1} \big\|P_{k}\big( Tf-f-P_{k-1}(Tf-f)\big)\big\|_{q,Q} \Big) \\
					&\leq C h_{\ell}^{p+1-|\bm{\alpha}|}\,|f|_{p+1,q,\tilde{Q}} + \frac{C}{h_{\ell}^{|\bm{\alpha}|}} [I].
				\end{align*}
				We estimate the sum $[I]$ by expressing $Q$ as a union of cells of level $k$ and repeatedly applying stability:
				\begin{align*}
					[I] &\leq \sum_{k=\ell+1}^{n-1} \|Tf-f-P_{k-1}(Tf-f)\|_{q,V_Q^k\cap\omega_k} \\
					&\leq \sum_{k=\ell+1}^{n-1} \Big( \|Tf-P_{k-1}(Tf)\|_{q,V_Q^k\cap\omega_k} + \|P_{k-1}f-f\|_{q,V_Q^k\cap\omega_k} \Big) \\
					&\leq\sum_{k=\ell+1}^{n-1} \|P_{k-1}f-f\|_{q,V_Q^k\cap\omega_{k-1}},
				\end{align*}
			where the last inequality results from $\omega_k\subset\omega_{k-1}$ and $Tf-P_{k-1}(Tf)=0$ in $\omega_{k-1}$ (cf. Theorem~\ref{teo: properties of P_ell} (iii)).
				Since $V_Q^k \subset \tilde{Q}$ for all $k \ge \ell+1$, Theorem~\ref{teo: properties of P_ell} (v) implies:
				\begin{equation*}
					[I] \leq \sum_{k=\ell+1}^{n-1} \|P_{k-1}f-f\|_{q,\tilde{Q}\cap\omega_{k-1}} \leq \sum_{k=\ell+1}^{n-1} C h_{k-1}^{p+1}|f|_{p+1,q,V_{\tilde{Q}\cap\omega_{k-1}}^{k-1}} \leq C h_{\ell}^{p+1}|f|_{p+1,q,V_{\tilde{Q}}^\ell},
				\end{equation*}
			where we bounded the geometric sum using $h_{k}=h_{k-1}/2$. Since $\tilde{Q} \subset V_{\tilde{Q}}^\ell$, we obtain
		\begin{equation*}
			[B] \leq C\: h_{\ell}^{p+1-|\bm{\alpha}|}\,|f|_{p+1,q,V_{\tilde{Q}}^\ell},
		\end{equation*}
		which, combined with \eqref{parteA}, proves the result for $\ell = 0$ with $N(Q) = V^0_{\tilde{Q}}$. If $\ell \geq 1$, we additionally have $V_{\tilde{Q}}^\ell \subset \widetilde{\parent(Q)}$, and the result follows with $N(Q) = \widetilde{\parent(Q)}$.
			\end{proof}

			We remark that \Cref{theo: sobolev norm estimation} holds under the sole assumption that the hierarchical mesh is weakly admissible, with a constant $C$ that depends only on $p$ and $d$. Local estimates of the same type were obtained in~\cite[Theorem 4]{S17} for general hierarchical meshes, with a constant that depends on the number of different levels of the THB-splines that do not vanish on each cell; a uniform bound therefore requires controlling this number, which is the case for admissible meshes (see \Cref{Def: Malla Admmisible} and~\cite[Corollary 1]{S17}). In this sense, our estimates are robust on the meshes produced by the refinement strategy of \Cref{S:algorithm}.

			These estimates hold only for cells of level $\ell$ contained in $\omega_\ell$. We conclude this section by extending them to all active cells: on an active cell $Q$ of level $\ell$ with approximation power $k$, the role of $h_\ell$ is played by $h_k$, which can thus be regarded as the effective mesh size of $Q$, denoted by $h_Q$ in \Cref{cor: estimates_active_cells_by_power}. Thus, the bound on each active cell reflects the local resolution attained by the hierarchical space.

			\begin{corollary}%
			\label{cor: estimates_active_cells_by_power}
			Let $\Pi$ be the quasi-interpolant defined in \eqref{def: operador Pi multinivel}, assume that $\QQQ$ is WAHM, and let $f \in W_{q}^{p+1}(\Omega)$ with $1\leq q\leq \infty$. Let $Q \in \mathcal{Q}_\ell$ be an active cell. Then,
			\begin{equation*}
				\|D^{\bm{\alpha}}(f-\Pi f)\|_{q,Q}\leq C\;h_Q^{p+1-|\bm{\alpha}|}|f|_{p+1,q, N(Q)},\qquad \forall\,\bm{\alpha},\, 0\leq |\bm{\alpha}|\leq p,
			\end{equation*}
			where the local mesh size $h_Q$ and the neighborhood $N(Q)$ are defined as $h_Q:=h_\ell$ and $N(Q) := \tilde{Q}$ if $Q \in \AAA_\ell(\QQQ)$, and $h_Q:=h_k$ and $N(Q) := \widetilde{\parent_{k-1}(Q)}$ if $Q \in \AAA_k(\QQQ)$ for some $1\leq k<\ell$. If $\ell\geq 1$ and $Q \in \AAA_0(\QQQ)$, we set $h_Q:=h_0$ and $N(Q):=V^0_{\widetilde{\parent_0(Q)}}$. The constant $C>0$ depends on $p$ and $d$, but is independent of $f$.
		\end{corollary}

		\begin{proof}
			The first case is a direct application of Theorem~\ref{theo: sobolev norm estimation} since $Q \subset \omega_\ell$. For the second case, in view of~\Cref{Def: poder de aproximacion} and~\Cref{obs: poder aproximacion activas}, let $Q' = \parent_k(Q) \subset \omega_k$ be its deactivated ancestor at level $k < \ell$, so that $Q'\subset\Omega_{k+1}$. Applying~\eqref{eq: estimate Sobolev} to $Q'$ yields:
			\begin{equation*}
				\|D^{\bm{\alpha}}(f-\Pi f)\|_{q,Q} \leq \|D^{\bm{\alpha}}(f-\Pi f)\|_{q,Q'} \leq C\;h_{k}^{p+1-|\bm{\alpha}|}|f|_{p+1,q,\widetilde{\parent(Q')}}.
			\end{equation*}
			for $k\geq 1$. If $k=0$, the same argument yields the bound with $V^0_{\tilde{Q}'}$ in place of $\widetilde{\parent(Q')}$.
			\end{proof}

	This result provides the theoretical motivation for our algorithm. Indeed, by partitioning the marked cells according to their approximation power (\Cref{alg:update_marked_elements}), we explicitly distinguish those for which the bound already involves the mesh size $h_\ell$ of their level from those for which it involves a coarser $h_k$ for some $k < \ell$.

\section{Experimental analysis of refinement strategies}\label{Sec: Test numerico}

In this section, we illustrate the behavior of the proposed refinement strategy through several numerical tests. Our main interest is to observe the gain of resolution that the refinement decided by the algorithm produces on the marked region. As references, we consider the plain refinement of the marked cells and the recursive refinement algorithm of~\cite[Algorithm 5]{BGV2018}, which refines the marked cells together with the additional cells required to preserve strict admissibility of class~$2$. We remark that these additional cells are selected in order to preserve the strict admissibility of the mesh, and not specifically to increase the resolution of the space on the marked region; hence, the enrichment they may produce there is a by-product of this requirement. In \Cref{subsec: un paso}, we analyze a single refinement step for the $L^2$-projection of a given function, which allows us to compare in detail the effect of each strategy on the marked region. In \Cref{subsec: ciclo completo}, we consider the complete adaptive loop for the Poisson equation, and we also include as a reference the marking of basis functions instead of cells.

We consider an adaptive loop of the form:
\begin{center}
	\resizebox{0.8\textwidth}{!}{
		\begin{tikzpicture}[
			node distance=0.8cm,
			block/.style={
				rectangle,
				draw=black,
				thick,
				fill=blue!5,
				minimum width=2.2cm,
				minimum height=1cm,
				font=\sffamily\small\bfseries,
				align=center
			},
			newblock/.style={
				block,
				fill=orange!20,
				draw=orange!80!black
			},
			arrow/.style={
				-{Stealth[scale=1.2]},
				thick
			}
			]

			\node[block] (solve) {SOLVE};
			\node[block, right=of solve] (estimate) {ESTIMATE};
			\node[block, right=of estimate] (mark) {MARK};
			\node[newblock, right=of mark] (update) {UPDATE\\MARK};
			\node[block, right=of update] (refine) {REFINE};

			\draw[arrow] (solve) -- (estimate);
			\draw[arrow] (estimate) -- (mark);
			\draw[arrow] (mark) -- (update);
			\draw[arrow] (update) -- (refine);

			\draw[arrow] (refine.south) -- ++(0,-0.8) -| (solve.south) ;

		\end{tikzpicture}
	}
\end{center}
We explain each module below:
\begin{itemize}
	\item SOLVE: The discrete solution is computed in the hierarchical space defined over the current mesh.

	\item ESTIMATE: Error estimators, denoted by $\mathcal{E}_Q$, are computed on the active cells $Q$ of the current mesh. Their specific form is given in each subsection.

	\item MARK: We select the active cells of the current hierarchical mesh where the estimators have the largest magnitude. Unless otherwise stated, we use the maximum strategy (MS) with a parameter $0<\theta<1$, which marks the active cells $Q\in\QQQ$ such that $\mathcal{E}_Q \geq \theta \max_{Q'\in\QQQ}\mathcal{E}_{Q'}$; larger values of $\theta$ lead to fewer marked cells.

	\item UPDATE MARK: The set of cells to be refined is determined from the marked ones by one of the following strategies:
	\begin{itemize}
		\item AP: Applies \Cref{alg: MARK_WEAKLY} (or equivalently, \Cref{alg:compute_marked_elements}), which selects the cells to be refined so that the approximation power of the hierarchical space increases on the marked cells, completing this selection, if necessary, so that the refined mesh is weakly admissible.
		\item MC+SA2: Applies the recursive marking algorithm presented in~\cite[Algorithm 5]{BGV2018}, which adds to the marked cells the cells that must also be refined to preserve the strictly admissible mesh structure of class~$2$.
		\item MC: No update is performed, that is, only the marked cells are refined. The resulting mesh need not be weakly admissible.
	\end{itemize}

	\item REFINE: The current hierarchical mesh is refined through the process described in \Cref{def: construccion de refinamiento}, using the cells determined in the previous step.
\end{itemize}

\subsection{A single refinement step}\label{subsec: un paso}

Let us compute the $L^2$-projection of a function $f$ onto the hierarchical spline space $\Span\mathcal{H}$. The approximation $s \in \Span\mathcal{H}$ is obtained by solving the standard linear system arising from the orthogonality condition: $\int_\Omega (s-f)g \, dx = 0,$ for all $g \in \Span\mathcal{H}$. As error estimators, we use the local errors $\mathcal{E}_Q = \|f-s\|_{L^2(Q)}$ on the active cells $Q$ of the hierarchical mesh.

We evaluate the performance of the strategies over the region $M$ of initially marked cells, defined as \(M := \bigcup \{Q\,\mid\,Q\in\MM_\ell,\ \ell=0,\dots,n-1\}\). Let $\HH_0$ and $\HH_1$ be the hierarchical bases before and after the refinement step, let $s_0\in\Span\HH_0$ and $s_1\in\Span\HH_1$ be the corresponding $L^2$-projections of $f$, and let $e_0=\|f-s_0\|_{L^2(M)}$ and $e_1=\|f-s_1\|_{L^2(M)}$. The error reduction index \(I_{\text{err}} := \log_{10}(e_0/e_1)\) roughly measures the digits of accuracy gained in $M$, and the efficiency index \(I_{\text{eff}} := I_{\text{err}}/\log_{10}\left(\#\HH_1/\#\HH_0\right)\) quantifies the error decay rate with respect to the number of degrees of freedom (DOFs). Finally, to see how many marked cells actually attain the full resolution of a higher level, we report the percentage PI of marked cells whose approximation power (\Cref{Def: poder de aproximacion}) increases after the refinement. Note that PI is a strict indicator: a marked cell does not contribute to it if even one of the B-splines needed to increase its approximation power is missing. By \Cref{Teo: princial refinamiento}, $\text{PI}=100\%$ for AP.

In each example, we start with an initial space of bicubic splines of maximum smoothness associated with a strictly admissible mesh of class~$2$ (and hence weakly admissible), where the hierarchical subdomains ${\Omega_\ell}$ are clustered (see \Cref{Def: Jerearquia subdominio agrupada}). Starting from the MARK module and concluding at ESTIMATE, we perform a single iteration of the adaptive loop with each of the strategies AP, MC+SA2 and MC.

\begin{example}\label{ej: funcion diagonal}
	Let us consider $f(x,y)= \arctan(25(x-y))$ over $\Omega = [0,1]^2$. The maximum strategy with $\theta = 0.7$ marks $142$ cells, all of them optimal cells of the finest level, which form the region $M$ shown in \Cref{fig:meshes_comparacion_diagonal}. With AP, all the necessary B-splines of the next level are added so that $M$ is contained in $\omega^*_{3}$, the corresponding subdomain of the refined mesh. With MC+SA2, only the marked cells are refined, since this is sufficient to preserve strict admissibility of class~$2$; hence, in this example MC+SA2 coincides with MC.

	\begin{figure}[h]
		\centering
		\includegraphics[width=.32\textwidth]{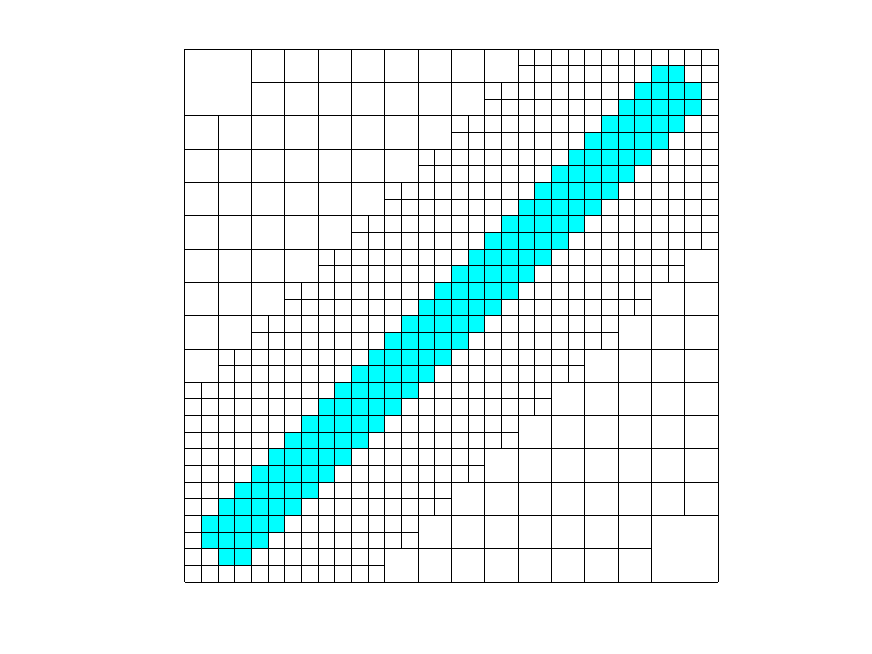}
		\includegraphics[width=.32\textwidth]{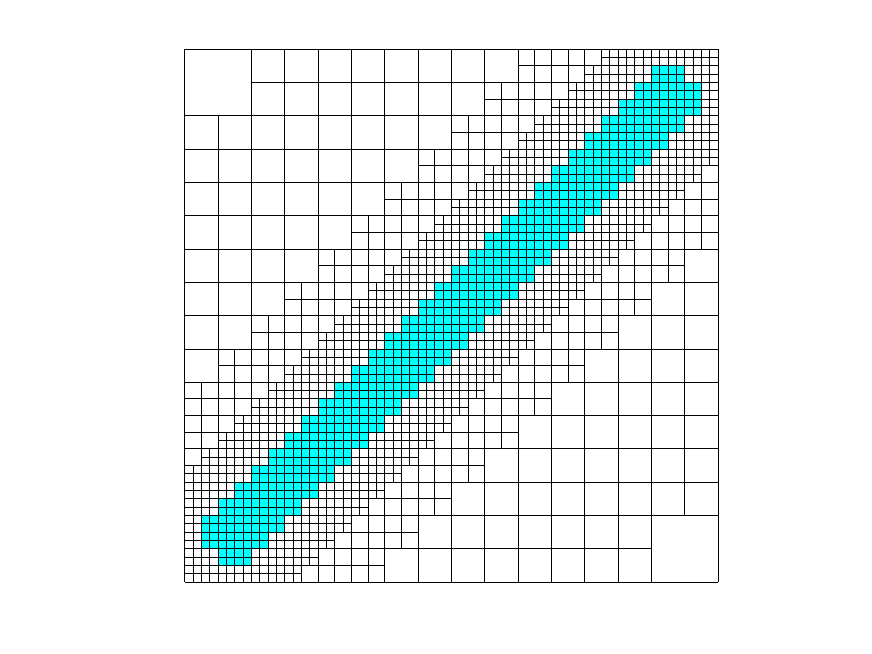}
		\includegraphics[width=.32\textwidth]{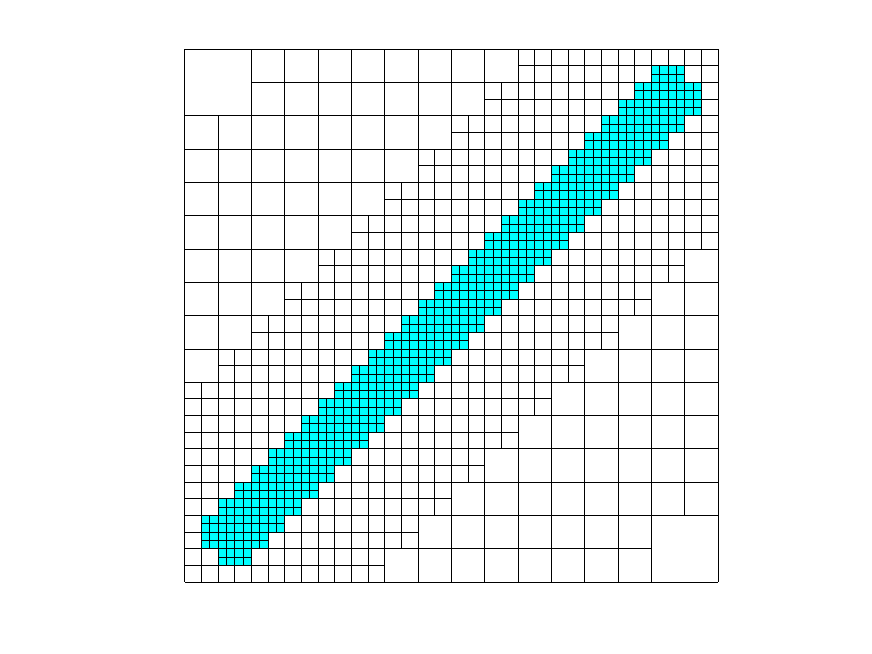}
		\caption{Initial hierarchical mesh (left) and refined meshes obtained with AP (center) and MC+SA2 (right), which coincides with MC, for \Cref{ej: funcion diagonal}. The shaded region indicates the initially marked cells.}
		\label{fig:meshes_comparacion_diagonal}
	\end{figure}

	\begin{table}[h]
		\small
		\centering
		\begin{tabular}{cccccc}
			\toprule
			Method & DOFs & $L^2(M)$-error & $I_{\text{eff}}$ & $I_{\text{err}}$ & PI (\%)\\
			\midrule
			--      & 757  & $1.68\times 10^{-3}$ & --   & --   & --    \\
			AP      & 1771 & $5.75\times 10^{-5}$ & 3.97 & 1.47 & 100.0 \\
			MC+SA2  & 974  & $6.73\times 10^{-4}$ & 3.63 & 0.40 & 0.0   \\
			MC      & 974  & $6.73\times 10^{-4}$ & 3.63 & 0.40 & 0.0   \\
			\bottomrule
		\end{tabular}
		\caption{Results of a single refinement step for \Cref{ej: funcion diagonal}: number of DOFs, error in the marked region, efficiency and error reduction indices, and percentage of marked cells whose approximation power increases (PI). The first row corresponds to the initial mesh.}
		\label{tabla: Error vs DOFs diagonal un paso}
	\end{table}

	\Cref{tabla: Error vs DOFs diagonal un paso} shows that, although MC adds $217$ DOFs, none of the marked cells attains the full resolution of the next level, whereas AP achieves it for all of them. As a result, AP gains about one additional digit of accuracy in $M$ ($I_{\text{err}}=1.47$ against $0.40$), with a comparable efficiency index.
\end{example}

\begin{example}\label{ej: funcion gauss}
	We consider $f(x,y)=\exp\left(-\frac{(x-0.5)^2+(y-0.5)^2}{2(0.1)^2}\right)$ over $\Omega = [0,1]^2$ and two different initial meshes, shown in \Cref{fig:meshes_comparacion_centro_05,fig:meshes_comparacion_centro_07}. In both cases, we use the maximum strategy with $\theta = 0.5$, which marks $32$ cells, predominantly optimal, in the first configuration, and $40$ cells, with a higher number of suboptimal cells, in the second one. The results are reported in \Cref{tabla:comparativa_mallas}.

	\begin{figure}[h]
		\centering
		\includegraphics[width=.24\textwidth]{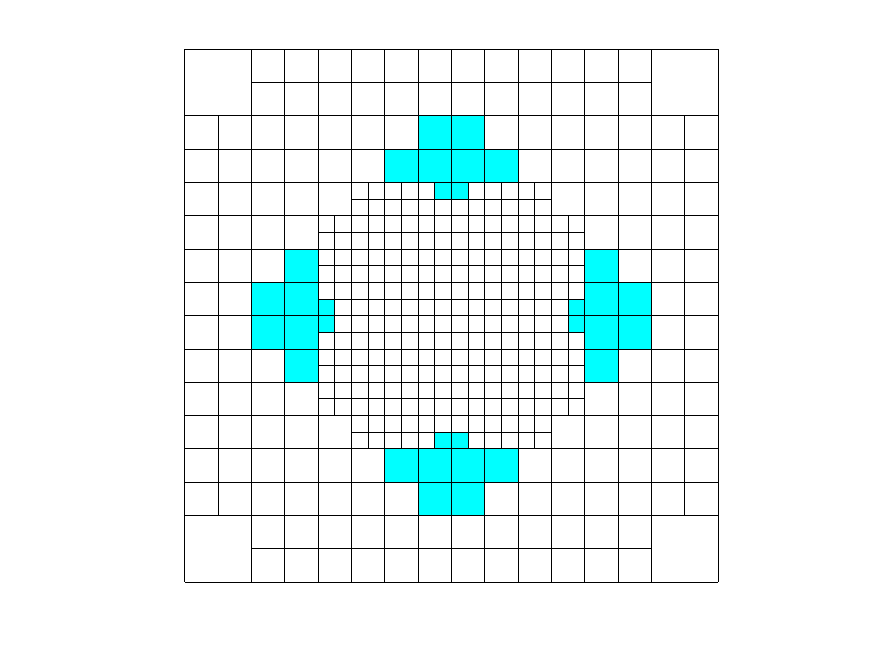}
		\includegraphics[width=.24\textwidth]{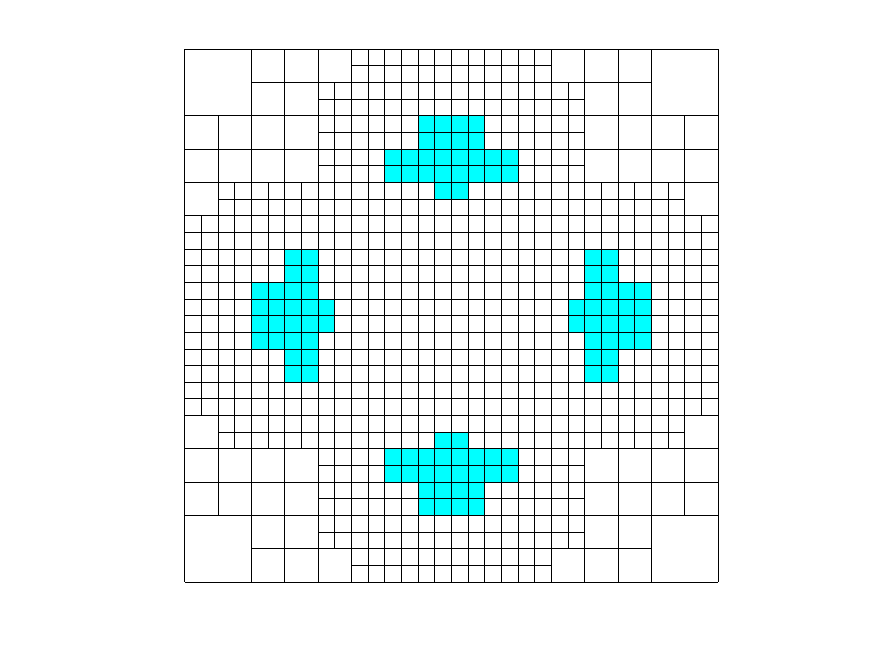}
		\includegraphics[width=.24\textwidth]{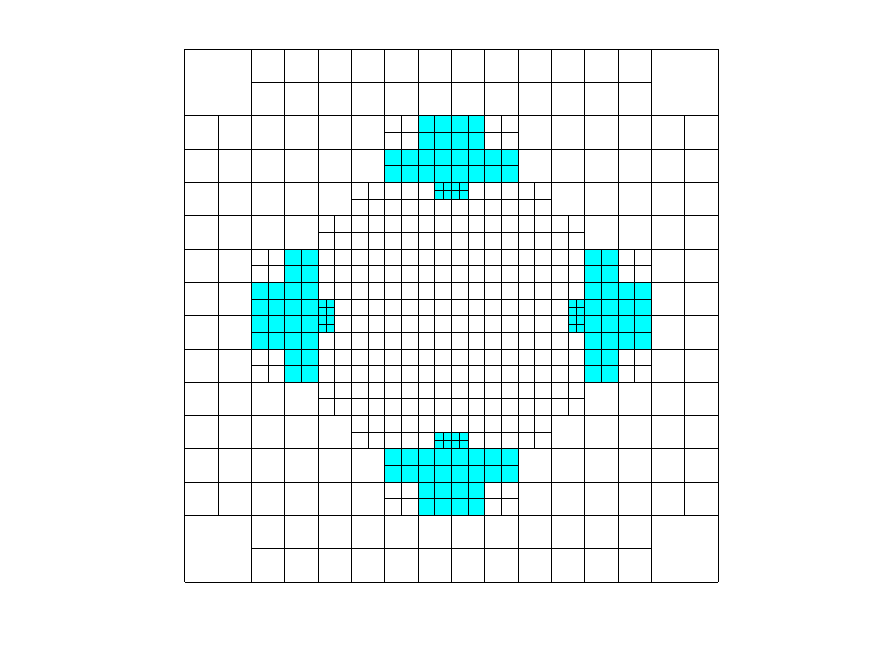}
		\includegraphics[width=.24\textwidth]{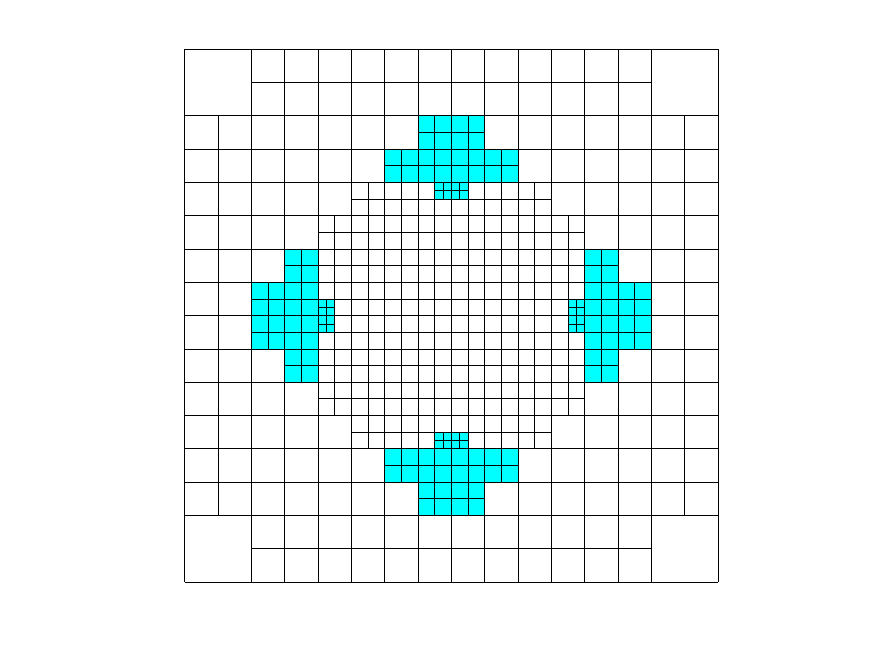}
		\caption{Initial hierarchical mesh and refined meshes obtained with AP, MC+SA2 and MC (from left to right) for the first configuration of \Cref{ej: funcion gauss}. The shaded region indicates the initially marked cells.}
		\label{fig:meshes_comparacion_centro_05}
	\end{figure}

	\begin{figure}[h]
		\centering
		\includegraphics[width=.24\textwidth]{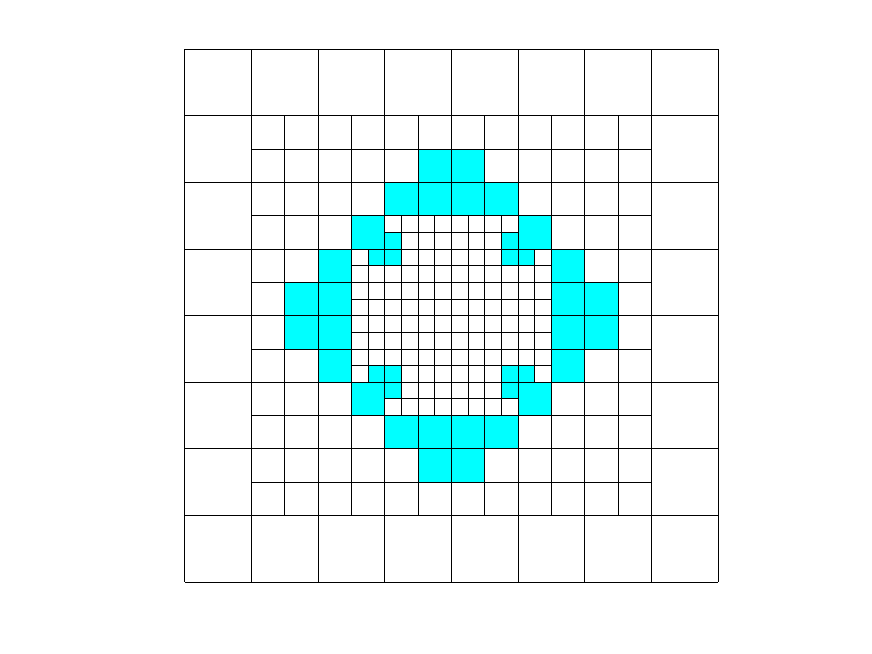}
		\includegraphics[width=.24\textwidth]{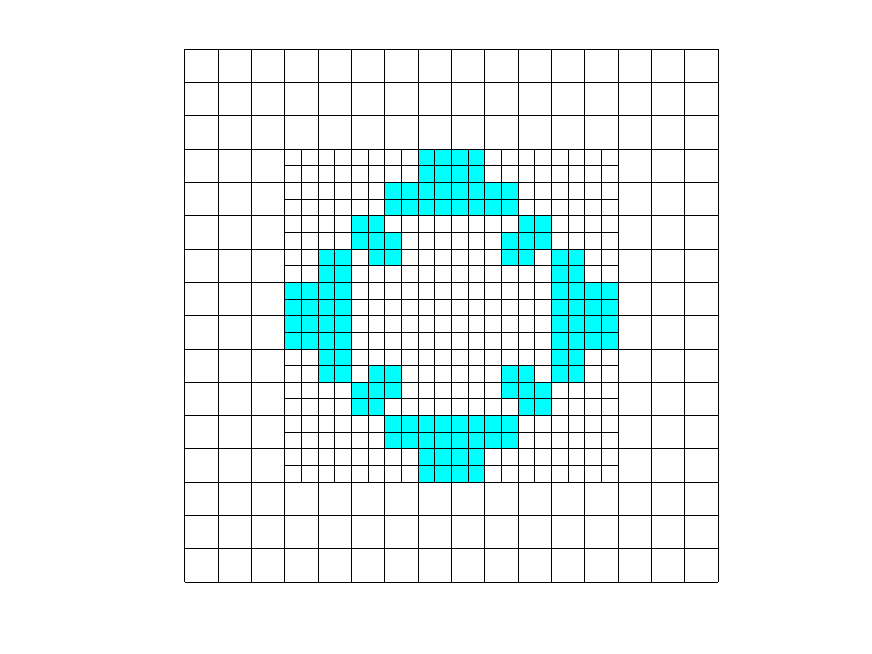}
		\includegraphics[width=.24\textwidth]{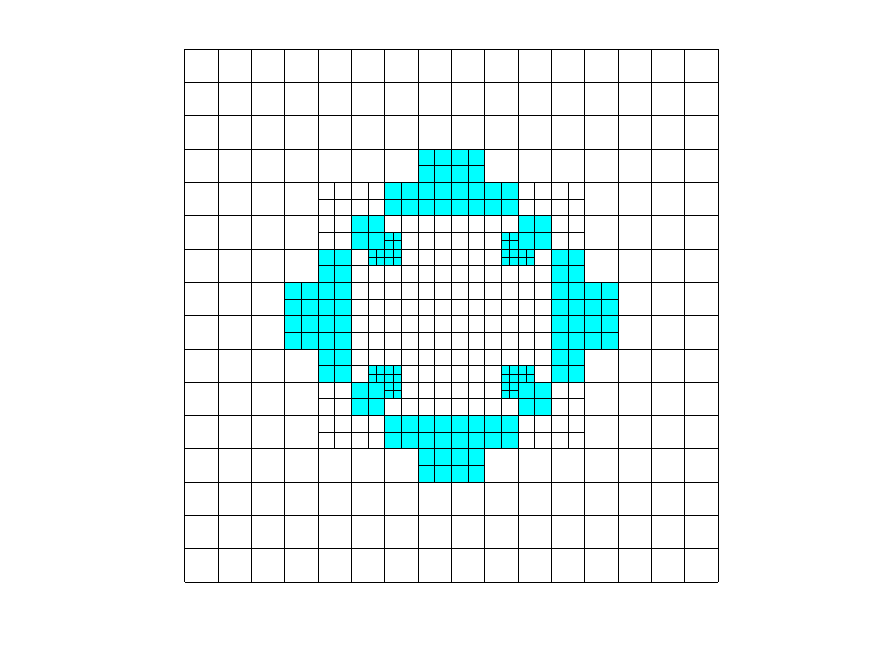}
		\includegraphics[width=.24\textwidth]{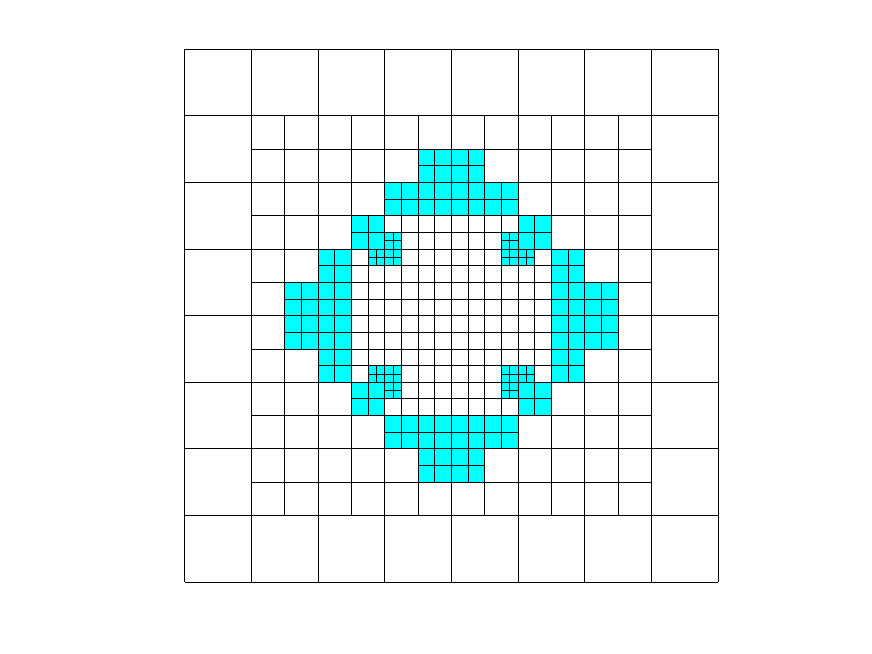}
		\caption{Initial hierarchical mesh and refined meshes obtained with AP, MC+SA2 and MC (from left to right) for the second configuration of \Cref{ej: funcion gauss}. The shaded region indicates the initially marked cells.}
		\label{fig:meshes_comparacion_centro_07}
	\end{figure}

	\begin{table}[h]
		\small
		\centering
		\begin{subtable}{\textwidth}
			\centering
			\begin{tabular}{cccccc}
				\toprule
				Method & DOFs & $L^2(M)$-error & $I_{\text{eff}}$ & $I_{\text{err}}$ & PI (\%)\\
				\midrule
				--      & 457 & $1.47\times 10^{-5}$ & --   & --   & --    \\
				AP      & 853 & $7.77\times 10^{-7}$ & 4.71 & 1.28 & 100.0 \\
				MC+SA2  & 529 & $4.56\times 10^{-6}$ & 7.99 & 0.51 & 25.0  \\
				MC      & 501 & $8.68\times 10^{-6}$ & 5.73 & 0.23 & 0.0   \\
				\bottomrule
			\end{tabular}
			\caption{First configuration (\Cref{fig:meshes_comparacion_centro_05}).}
			\label{tab:comparativa_mallas_a}
		\end{subtable}

		\medskip
		\begin{subtable}{\textwidth}
			\centering
			\begin{tabular}{cccccc}
				\toprule
				Method & DOFs & $L^2(M)$-error & $I_{\text{eff}}$ & $I_{\text{err}}$ & PI (\%)\\
				\midrule
				--      & 253 & $4.81\times 10^{-5}$ & --   & --   & --    \\
				AP      & 601 & $5.40\times 10^{-6}$ & 2.53 & 0.95 & 100.0 \\
				MC+SA2  & 513 & $1.09\times 10^{-5}$ & 2.11 & 0.65 & 90.0  \\
				MC      & 309 & $1.92\times 10^{-5}$ & 4.58 & 0.40 & 0.0   \\
				\bottomrule
			\end{tabular}
			\caption{Second configuration (\Cref{fig:meshes_comparacion_centro_07}).}
			\label{tab:comparativa_mallas_b}
		\end{subtable}
		\caption{Results of a single refinement step for \Cref{ej: funcion gauss}: number of DOFs, error in the marked region, efficiency and error reduction indices, and percentage of marked cells whose approximation power increases (PI). The first row of each table corresponds to the initial mesh.}
		\label{tabla:comparativa_mallas}
	\end{table}

	In the first configuration (\Cref{tab:comparativa_mallas_a}), AP reduces the error in $M$ by $1.28$ digits, compared with $0.51$ for MC+SA2 and $0.23$ for MC, and it is the only strategy that increases the approximation power of all the marked cells. MC+SA2 attains the highest efficiency index, since it adds only $72$ DOFs, whereas AP adds $396$; recall that the aim of AP is to increase by at least one the approximation power of every marked cell, which, for an optimal cell, requires adding all the B-splines of the next level that act on it.

	In the second configuration (\Cref{tab:comparativa_mallas_b}), AP attains both the largest error reduction and a higher efficiency index than MC+SA2. In this case, MC+SA2 increases the approximation power of $90\%$ of the marked cells, compared with $25\%$ in the first configuration. This difference can be explained as follows. In a strictly admissible mesh of class~$2$, every refined cell of level $\ell$ is contained in $\omega_\ell$. Hence, when MC+SA2 refines a suboptimal marked cell, which is not contained in $\omega_\ell$, the completion also refines the cells of level $\ell-1$ needed to include it in $\omega_\ell$, and its approximation power increases. In contrast, the approximation power of an optimal marked cell increases only if all its children are contained in $\omega_{\ell+1}$, and strict admissibility of class~$2$ does not require this. Since the second configuration contains more suboptimal marked cells, MC+SA2 increases the approximation power of a larger fraction of them. Thus, the enrichment of the marked region produced by MC+SA2 depends on the configuration of the marked cells, while AP increases the approximation power of all of them by construction (see \Cref{Teo: princial refinamiento}). Finally, MC does not increase the approximation power of any marked cell in either configuration, as in \Cref{ej: funcion diagonal}. Moreover, in this example it yields the smallest error reduction, and its efficiency index is relatively high only because it adds very few DOFs ($44$ and $56$, respectively).
\end{example}

\subsection{The complete adaptive loop}\label{subsec: ciclo completo}

Let us consider the Poisson equation with Dirichlet boundary conditions given by
\begin{equation}\label{eq: ecuacion Poisson}
	\begin{cases}
		-\Delta u = f \quad &\text{in}\; \Omega \\
		u= g \quad &\text{on}\; \partial \Omega
	\end{cases}
\end{equation}
where $f$ and $g$ are given and $\Omega=[0,1]^2$. We consider a Galerkin discretization in a hierarchical spline space to approximate the weak solution $u$ of~\eqref{eq: ecuacion Poisson}, and we use the element-based residual error estimators $\mathcal{E}_Q$ introduced in~\cite{BGi15}.

Besides MS, we use D\"orfler's strategy, which marks a set $\mathcal{M}$ of active cells of minimal cardinality such that $\sum_{Q\in\mathcal{M}}\mathcal{E}_Q^2 \geq \theta^2 \sum_{Q\in\QQQ}\mathcal{E}_Q^2$; in this case, larger values of $\theta$ lead to more marked cells. In addition to AP, MC+SA2 and MC, we consider as a reference the strategy MF, in which basis functions are marked instead of cells: the function-based error estimators of~\cite{BG18_aposteriori} are computed, one for each active basis function, the marking strategy is applied to them, and, for each marked basis function, the active cells of its level contained in its support are refined.

\begin{example}\label{ej: singular interior}
	Let $f$ and $g$ be such that $u(x,y)=r^{3/2}$, with $r=|(x,y)-(0.5,0.5)|$, is the solution of~\eqref{eq: ecuacion Poisson}, that is, $f=-\tfrac{9}{4}r^{-1/2}$ and $g=u|_{\partial\Omega}$. Since $u\in H^{5/2-\varepsilon}(\Omega)$ for all $\varepsilon>0$, uniform refinement yields a convergence rate of order $\#\text{DOFs}^{-3/4}$ in the energy norm, whereas the optimal rate for bicubic splines is $\#\text{DOFs}^{-3/2}$. We consider a bicubic tensor-product spline space of maximum smoothness whose initial mesh consists of $8\times 8$ cells, so that the singular point is a vertex of the mesh at every level, and we perform at most $12$ iterations of the adaptive loop. In order to analyze the influence of the number of marked cells, we use MS with $\theta \in \{0.25, 0.5, 0.75, 0.9\}$ and D\"orfler's strategy with $\theta \in \{0.9, 0.7, 0.5, 0.3\}$; in both cases, fewer cells are marked along these sequences.

	\begin{figure}[h]
		\centering
		\includegraphics[width=\textwidth]{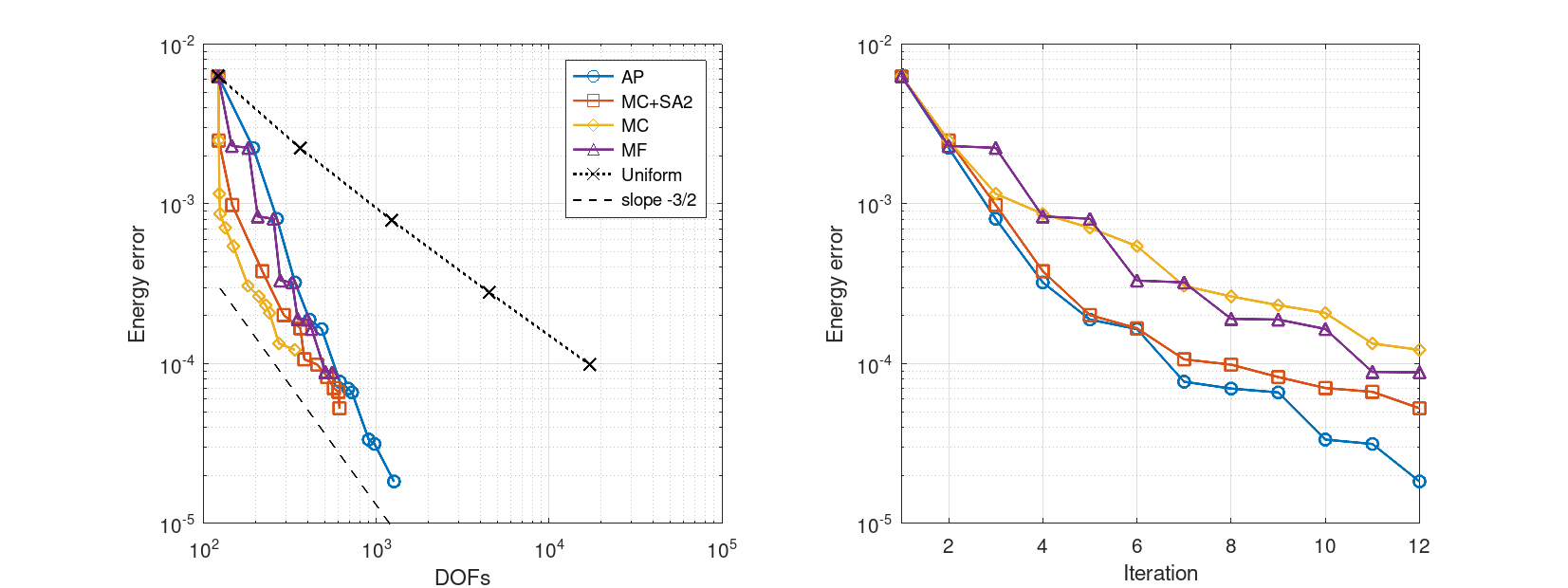}
		\caption{Energy error versus DOFs (left) and versus iterations (right) for \Cref{ej: singular interior}, with MS and $\theta=0.9$. The dotted line corresponds to uniform refinement and the dashed line indicates the optimal slope.}
		\label{fig: curvas singular}
	\end{figure}

	\Cref{fig: curvas singular} shows the results for MS with $\theta=0.9$, a setting in which few cells are marked at each iteration. All the strategies exhibit the optimal convergence rate with respect to the number of DOFs, whereas uniform refinement attains an error of $9.9\times10^{-5}$ with 17161 DOFs. However, they progress very differently along the iterations. During the first iterations, AP and MC+SA2 attain similar errors, but then MC+SA2, as well as MF and MC, slows down considerably: after $12$ iterations, AP attains an error below $2\times10^{-5}$, whereas MC+SA2, MF and MC stop at $5.3\times10^{-5}$, $8.8\times10^{-5}$ and $1.2\times10^{-4}$, respectively. In other words, when few cells are marked, the other strategies advance very slowly along the optimal error--DOFs curve. \Cref{fig: mallas singular} shows the meshes obtained after $12$ iterations in this setting. At this point, AP has 1261 DOFs, whereas MC+SA2, MF and MC have only 610, 553 and 336 DOFs, respectively. With AP, the refinement around the singular point is wide and gradually graded: the whole displayed region has been refined at least twice, and each level covers a large neighborhood of the region refined at the next level. In contrast, with MC most of the displayed region remains at coarser levels and the finest levels consist of a few cells concentrated around the singular point, which is consistent with the slow decrease of the error along the iterations observed in \Cref{fig: curvas singular}. MC+SA2 and MF show an intermediate behavior.

	\begin{figure}[h]
		\centering
		\includegraphics[width=.24\textwidth]{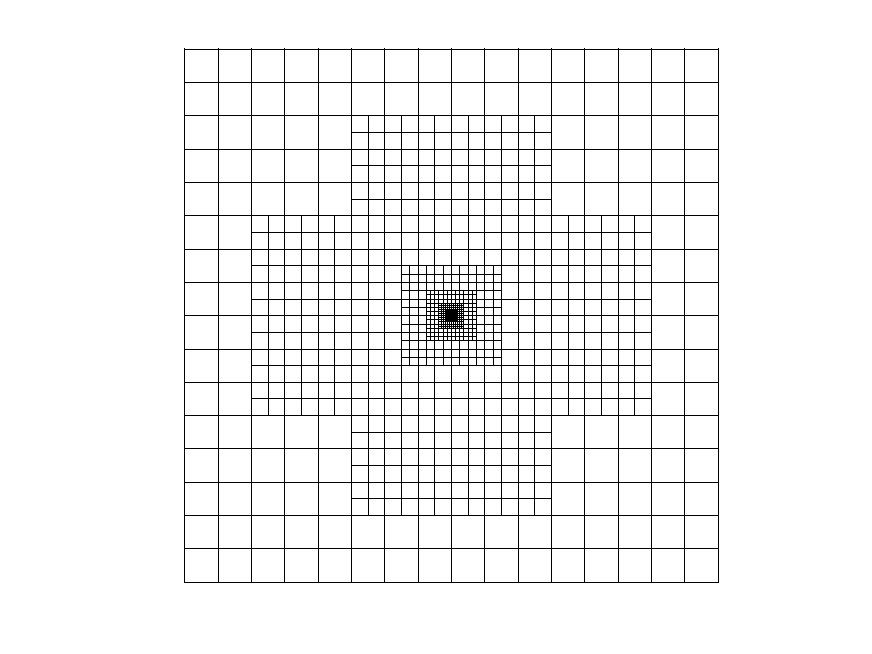}
		\includegraphics[width=.24\textwidth]{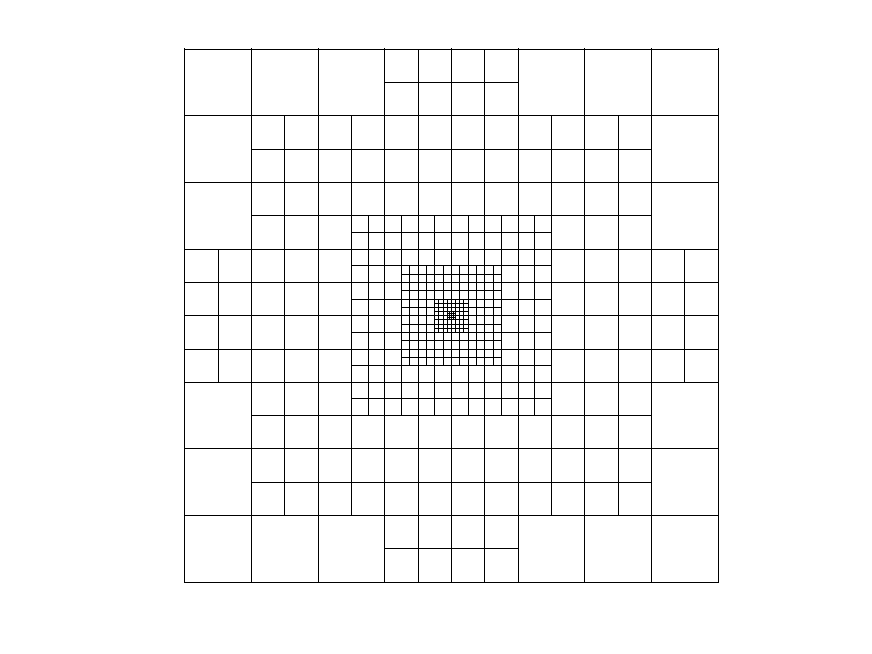}
		\includegraphics[width=.24\textwidth]{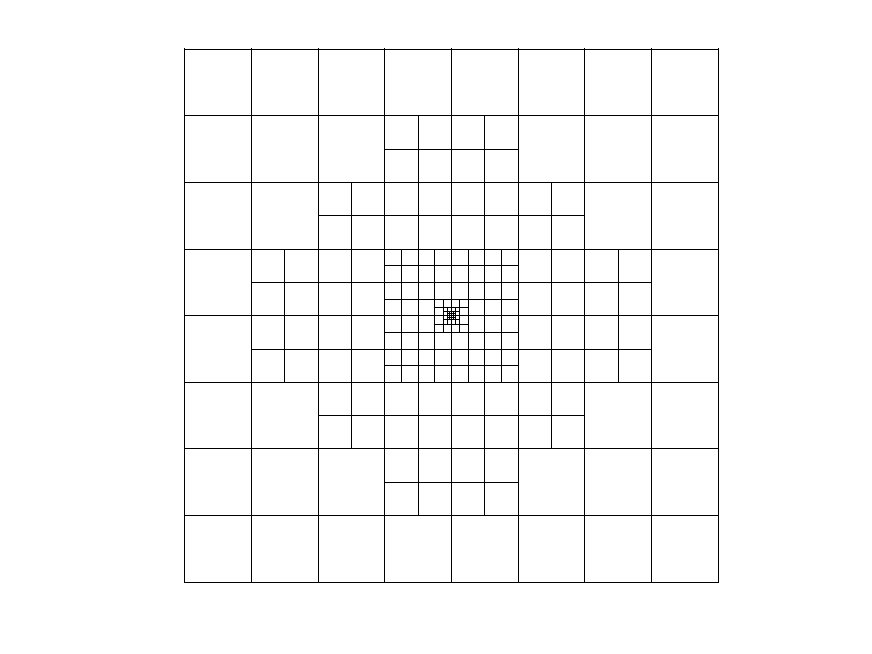}
		\includegraphics[width=.24\textwidth]{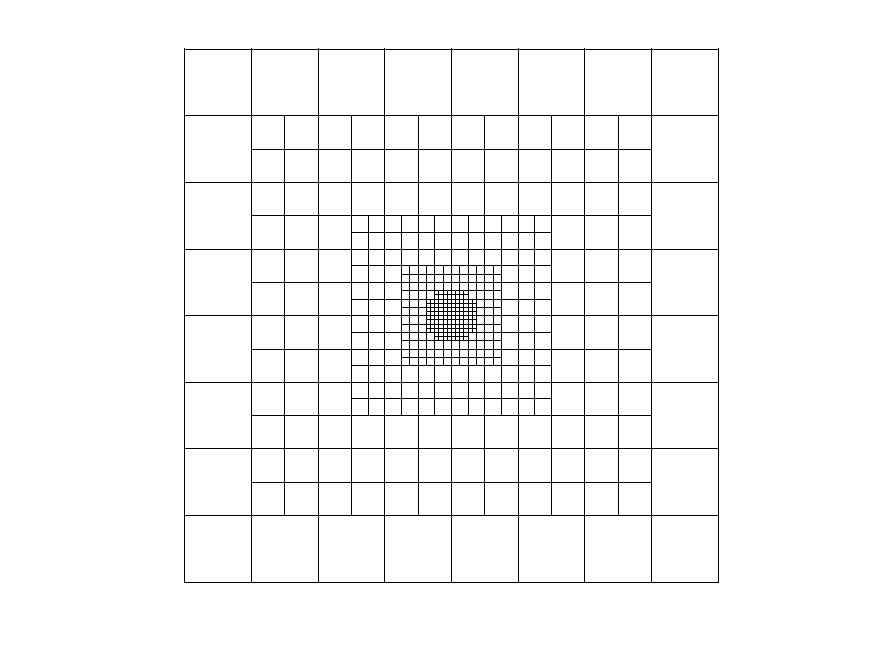}
		\caption{Meshes after $12$ iterations for \Cref{ej: singular interior}, with MS and $\theta=0.9$, restricted to $[0.25,0.75]^2$: AP, MC+SA2, MC and MF (from left to right).}
		\label{fig: mallas singular}
	\end{figure}

	\Cref{tab: robustness} summarizes the results for all the marking settings. It reports the number of iterations and DOFs needed to reach the tolerance $2\times10^{-5}$, chosen as the smallest value, rounded up, that AP attains in all the settings within $12$ iterations. When many cells are marked (MS with $\theta=0.25$ or $0.5$, and D\"orfler's strategy with $\theta=0.9$), all the strategies reach the tolerance in a similar number of iterations, between $7$ and $11$. As fewer cells are marked, AP needs more iterations, up to $12$, but it is the only strategy that reaches the tolerance in all the settings: with MS and $\theta\geq 0.75$, and with D\"orfler's strategy and $\theta\leq 0.5$, none of the other strategies reaches it. In these settings, they add few degrees of freedom per iteration; for instance, with D\"orfler's strategy and $\theta=0.3$, MC adds on average $7.9$ DOFs per iteration and attains an error of $2.6\times10^{-4}$ after $12$ iterations. Whenever several strategies reach the tolerance, AP never needs more iterations than the others, and the number of DOFs is of the same order: AP uses between $0.7$ and $1.7$ times the DOFs of the other strategies.

	For comparison, \Cref{tab: robustness} also includes the results obtained with the same initial space and settings for the solution $u(x,y)=\arctan(25(x-y))$, which has a sharp interior layer, with tolerance $4\times10^{-3}$ chosen in the same way; uniform refinement requires 17161 DOFs to reach it. In this case the differences are more pronounced. AP reaches the tolerance in $5$ or $6$ iterations in all the settings, except for D\"orfler's strategy with $\theta=0.3$, where it needs $11$. MC and MC+SA2 reach it only when many cells are marked, and otherwise make little progress: with D\"orfler's strategy and $\theta=0.3$, their errors after $12$ iterations are $2.81$ and $1.47$, respectively. MF reaches the tolerance in all the settings with MS, but not with D\"orfler's strategy and $\theta\leq 0.5$.

	\begin{table}[h]
		\footnotesize
		\centering
		\setlength{\tabcolsep}{4pt}
		\begin{tabular}{lcccccccc}
			\toprule
			& \multicolumn{4}{c}{MS} & \multicolumn{4}{c}{D\"orfler} \\
			\cmidrule(lr){2-5}\cmidrule(lr){6-9}
			Method & $\theta=0.25$ & $\theta=0.5$ & $\theta=0.75$ & $\theta=0.9$ & $\theta=0.9$ & $\theta=0.7$ & $\theta=0.5$ & $\theta=0.3$ \\
			\midrule
			\multicolumn{9}{l}{$u=r^{3/2}$, tolerance $2\times 10^{-5}$}\\
			AP     & 8 (1525) & 9 (1597)  & 11 (1261) & 12 (1261) & 7 (1645) & 8 (1645)  & 11 (1393) & 12 (1261) \\
			MC+SA2 & 9 (1698) & 10 (1122) & --        & --        & 8 (1518) & 10 (1174) & --        & --        \\
			MC     & 9 (1615) & 11 (1312) & --        & --        & 8 (1478) & --        & --        & --        \\
			MF     & 8 (1381) & 9 (1381)  & --        & --        & 7 (1309) & 10 (1189) & --        & --        \\
			\midrule
			\multicolumn{9}{l}{$u=\arctan(25(x-y))$, tolerance $4\times 10^{-3}$}\\
			AP     & 5 (4051) & 5 (3817)  & 5 (3457)  & 5 (3289)  & 5 (3817) & 5 (3271)  & 6 (3149)  & 11 (3129) \\
			MC+SA2 & 7 (5171) & 8 (3019)  & --        & --        & 7 (3768) & 12 (3011) & --        & --        \\
			MC     & 7 (5923) & 9 (2293)  & --        & --        & 8 (5157) & --        & --        & --        \\
			MF     & 5 (3913) & 5 (2995)  & 8 (3109)  & 11 (2923) & 5 (3151) & 7 (2707)  & --        & --        \\
			\bottomrule
		\end{tabular}
		\caption{Number of iterations and, in parentheses, number of DOFs needed to reach the tolerance, for MS and D\"orfler's strategy with different values of $\theta$. Within each marking strategy, fewer cells are marked from left to right. A dash indicates that the tolerance was not reached within $12$ iterations. The tolerances are the smallest values, rounded up, that AP attains in all the settings.}
		\label{tab: robustness}
	\end{table}

	These results show that the main advantage of AP is its robustness with respect to the choice of the marking strategy and its parameter. Since AP increases by at least one the approximation power of each marked cell, no matter how many there are, it keeps reducing the error at every iteration, whereas the other strategies may stagnate when this number is small.
\end{example}

\section{Conclusions}\label{S:conclusions}

We have proposed an adaptive refinement strategy for hierarchical splines in which the resolution added to the space is explicitly targeted at the cells where it is required. The strategy is driven by the approximation power of the marked cells: it refines enough so that, on each marked cell, the hierarchical space locally reproduces the full tensor-product spline space of one more level than before, that is, the approximation power of every marked cell increases by at least one (\Cref{Teo: princial refinamiento}). This can be carried out with the standard refinement of hierarchical meshes, in which each cell is subdivided at most once per step, provided that the subdomains where the hierarchical space attains the full resolution of each level are nested. This leads us to work with weakly admissible hierarchical meshes, which we have characterized in terms of the cell neighborhoods $\mathcal{N}_Q$ of \Cref{Def: C_Q} (see \Cref{Teo: 1ra equivalencia MDA,Equivalencia_Mallas_DA}). Moreover, for clustered hierarchies, the refinement is well defined and produces again a mesh of the same class (\Cref{Lemma: buena definicion de la malla} and \Cref{Teo: princial refinamiento}). Finally, the total number of cells created along the adaptive process is at most proportional to the total number of marked cells, with a constant that depends only on the dimension and the degree (\Cref{teo: complejidad de refinamiento}).

From the approximation point of view, we have extended the local $L^q$ estimates for the multilevel quasi-interpolant of~\cite{K98,BG15} to Sobolev seminorms on weakly admissible meshes, with constants depending only on the degree and the dimension (\Cref{theo: sobolev norm estimation}). As a consequence, the local error on an active cell is governed by its approximation power (\Cref{cor: estimates_active_cells_by_power}), which provides the theoretical justification for using this quantity as the criterion that drives the refinement.

The numerical experiments illustrate these properties. In a single refinement step, the proposed strategy increases the approximation power of all the marked cells and produces the largest error reduction on the marked region, whereas refining only the marked cells may add degrees of freedom without increasing the approximation power of any of them, and the completion required by strict admissibility increases it only as a side effect, depending on the configuration of the marked cells. In the complete adaptive loop, all the strategies attain the optimal convergence rate; the proposed strategy adds more degrees of freedom at each step, but it reaches a given accuracy in fewer or the same number of iterations, with a number of DOFs of the same order, and it is considerably less sensitive to the choice of the marking strategy and parameter, whereas the other strategies may stagnate when few cells are marked.

The selection of cells based on the approximation power can also be combined with other classes of hierarchical meshes, by refining further to recover the desired structure, as long as the resulting meshes are weakly admissible and associated with clustered hierarchies.

The convergence and optimality analysis of adaptive methods with hierarchical splines usually relies on admissibility conditions; an interesting question for future work is whether similar results can be established for the adaptive loop based on the proposed strategy on weakly admissible meshes. It would also be of interest to clarify the relation between weak admissibility and admissibility of class $m>2$, and to assess the performance of the strategy in three-dimensional problems.

	\section*{Acknowledgements}
	This work was partially supported by Universidad Nacional del
	Litoral through grant CAI+D-2024 85520240100018LI, and by Consejo Nacional de Investigaciones Científicas y Técnicas through grant PIP 11220250101092CO (Argentina). This support is gratefully acknowledged.

\def\cprime{$'$}

\end{document}